%% file: ex_arxiv.tex
\documentclass[hidelinks,onefignum, onetabnum]{siamart220329}

\input{preamble}

\usepackage{lipsum}
\usepackage{amsfonts}
\usepackage{graphicx}
\usepackage{epstopdf}
\usepackage{algorithmicx}

\ifpdf
  \DeclareGraphicsExtensions{.eps,.pdf,.png,.jpg}
\else
  \DeclareGraphicsExtensions{.eps}
\fi

\newsiamremark{remark}{Remark}
\newsiamremark{hypothesis}{Hypothesis}
\crefname{hypothesis}{Hypothesis}{Hypotheses}
\newsiamthm{claim}{Claim}

\usepackage{amsopn}

\ifpdf
\hypersetup{
  pdftitle={Sharper Convergence Guarantees for Federated Learning with Partial Model Personalization},
  pdfauthor={Yiming Chen}
}
\fi

\begin{document}

\headers{Sharper Convergence Guarantees for Federated Learning}{Y. Chen, L. Cao, K. Yuan, and Z. Wen}

\title{Sharper Convergence Guarantees for Federated Learning with Partial Model Personalization}

\author{Yiming Chen\thanks{Beijing International
Center for Mathematical Research, Peking University, Beijing, China 
(\email{abcdcym@stu.pku.edu.cn},  \email{caoliyuan@bicmr.pku.edu.cn}, \email{wenzw@pku.edu.cn}).}
\and Liyuan Cao$^*$
\and Kun Yuan\thanks{Corresponding author. Center for Machine Learning Research, Peking Univeristy, Beijing, China (kunyuan@pku.edu.cn).  } \and Zaiwen Wen$^*$
}

\maketitle

\begin{abstract}
Partial model personalization, which encompasses both shared and personal variables in its formulation, is a critical optimization problem in federated learning. It balances individual client needs with collective knowledge utilization, and serves as a general formulation covering various key scenarios, ranging from fully shared to fully personalized federated learning. This paper introduces two effective algorithms, FedAvg-P and Scaffold-P, to solve this problem and provides sharp convergence analyses, quantifying the influence of gradient variance, local steps, and partial client sampling on their performance. Our established rates surpass existing results and, meanwhile, are based on more relaxed assumptions. Additionally, our analyses are also applicable to fully shared or fully personalized federated learning, matching or even outperforming their best known convergence rates. Numerical experiments corroborate our theoretical findings. 
\end{abstract}

\begin{keywords}
federated learning, non-convex optimization, stochastic optimization
\end{keywords}

\begin{AMS}
  68Q25, 90C15, 90C26
\end{AMS}

\section{Introduction}\label{sec-introduction}
Federated learning (FL) \cite{konevcny2016federated,mcmahan2017communication,kairouz2021advances} is a powerful paradigm for distributed machine learning that enables collaborative model training across multiple clients while preserving data privacy by avoiding transferring raw data to a central server. 
Conventionally, FL seeks to enhance a global shared model by aggregating local model updates from each client. However, this traditional approach faces challenges with the intrinsic heterogeneity in data distributions across different clients, rendering it difficult to find one single shared model that caters to the personalized needs of every client.
 
Personalized federated learning (PFL) \cite{li2021ditto,mushtaq2021spider,t2020personalized,fallah2020personalized,pillutla2022federated,collins2021exploiting} has emerged as a promising approach to address this limitation. It extends the traditional FL framework by introducing personalized components, thereby enabling each client to maintain a local model that can be trained according to their specific requirements. This paper  investigates a PFL setting in which $n$ clients collaborate to solve the following partial model personalization problem \cite{pillutla2022federated,collins2021exploiting}:
\begin{align}\label{prob-general}
\min_{u, \vv} \quad f(u, \vv) := \frac{1}{n}\sum_{i=1}^n f_i(u, v_i) \quad \mbox{where} \quad f_i(u, v_i) = \EE_{\xi_i \sim \mathcal{D}_i}[F(u, v_i; \xi_i)].
\end{align}
This optimization problem involves two groups of variables: the shared variable $u \in \mathbb{R}^{d_u}$, which is common to all clients and captures the shared model structure, and the personal variables $\vv = [v_1; \cdots; v_n] \in \mathbb{R}^{n d_v}$, where each $v_i \in \mathbb{R}^{d_v}$ is exclusively maintained by each client $i$ and represents its own personal model component. For instance, in a model proposed in \cite{collins2021exploiting}, variable $u$ denotes a shared data representation across all clients, while $v_i$ indicates the local label head unique to client $i$. 

Partial model personalization problem \eqref{prob-general} is critical to PFL. With both shared and personal variables in its formulation, the solution to this problem caters for the individual requirements of each client while still leveraging the collective knowledge across all clients. 
Additionally, problem \eqref{prob-general} also serves as a general formulation that encompasses various important scenarios in FL. Without all personal variables, this problem simplifies into the fully shared FL formulation with one single global model \cite{konevcny2016federated, mcmahan2017communication, kairouz2021advances} (see problem \eqref{prob-full}). Without the shared variables, it reduces to the fully personalized model \cite{chayti2021linear} (see problem \eqref{prob-full-partial}). Notably, problem \eqref{prob-general} also covers PFL models based on regularizations  \cite{li2021ditto, mushtaq2021spider, t2020personalized} as well as on interpolations \cite{deng2020adaptive}.

Partial model personalization problem \eqref{prob-general} poses several significant challenges to algorithmic development, including limited communication bandwidth between clients and the server, time-varying client participation, heterogeneous data distributions, and a complex optimization structure with shared and personal variables coupled together. To tackle these challenges, FedRep \cite{collins2021exploiting} focuses on a simplified linear regression problem and proposes a collaborative approach, which enables the server and all clients to optimize the shared variable $u$ while allowing each client $i$ to learn its personal variable $v_i$ individually. FedSim and FedAlt \cite{pillutla2022federated}
extend such collaborative approach to the general non-convex optimization problem. In FedSim, both shared and personal variables are updated simultaneously, while in FedAlt, they are updated in a Gauss-Seidel manner. A recent work \cite{huang2022stochastic} studies problem \eqref{prob-general} in a more sophisticated decentralized collaborative setting. 

While the above algorithms have shown empirical success, their convergence analyses are limited in several ways. First, specific algorithms proposed in \cite{collins2021exploiting,pillutla2022federated} rely on the assumption of bounded data heterogeneity. Their convergence analyses become invalid when this assumption is violated. Second, the analyses in  \cite{collins2021exploiting,pillutla2022federated,huang2022stochastic} are not precise enough to fully capture the impact of important factors such as the number of clients, the number of local steps, and the partial client participation strategy on algorithmic performances. 
Third, the analyses in \cite{collins2021exploiting,pillutla2022federated,huang2022stochastic} are not sufficiently general to encompass the best known convergence rates in special scenarios. For instance, removing all personal variables from the analyses in \cite{collins2021exploiting,pillutla2022federated,huang2022stochastic} leads to convergence rates  for the traditional FL problem with a single shared model. Yet, these rates are notably worse than the best-known ones established in \cite{karimireddy2020scaffold,yang2021achieving,yu2019parallel}.

\subsection{Contributions} 
This paper introduces improved algorithms for solving problem~\eqref{prob-general}, and provides sharp convergence guarantees that overcome the limitations discussed above. Specifically, our key contributions are outlined below.
\vspace{1mm}
\begin{itemize}[leftmargin=20pt]
\item 
We adapt FedAvg \cite{konevcny2016federated,mcmahan2017communication} and Scaffold \cite{karimireddy2020scaffold}, two prominent algorithms designed for the traditional FL setting with a single shared model, to address problem \eqref{prob-general} which allows each client to maintain a personalized model. 
This results in two new algorithms, FedAvg-P and Scaffold-P. Notably, Scaffold-P can effectively solve problem \eqref{prob-general} without assuming bounded data heterogeneity across clients.

\vspace{1mm}
\item We present rigorous convergence analyses for both FedAvg-P and Scaffold-P algorithms. Our established convergence rates improve upon existing results for problem \eqref{prob-general}, offering a comprehensive understanding of how gradient variance, the number of clients, the number of local steps, and the partial client participation strategy influence the algorithmic performances.

\vspace{1mm}
\item Our analyses are highly versatile, encompassing or even surpassing existing state-of-the-art convergence rates when problem \eqref{prob-general} is reduced to certain special scenarios. For example, in the fully shared FL formulation with one single global model, our analyses yield convergence rates that is sharper than all existing results. Furthermore, our analyses establish that the FedAvg method, when all clients participate in the optimization procedure, can achieve convergence without relying on any assumption of bounded data heterogeneity. This is a novel finding, since all prior analyses of vanilla FedAvg require the bounded heterogeneity assumption when they do not make any algorithmic adjustments.
\end{itemize}

\subsection{Other related work} 
There has been extensive research on the classical FL model with one single shared variable. References \cite{li2019convergence, qu2021federated} have investigated the convergence properties of FedAvg in convex settings, while \cite{yang2021achieving} has demonstrated that FedAvg can achieve linear speedup in non-convex settings. These works clarified that data heterogeneity across clients can significantly slow down the convergence of FedAvg. In response to this challenge, a study by Huang et al. \cite{huang2023distributed} investigates FedAvg under a relaxed assumption of bounded data heterogeneity in terms of objective function values. More advanced algorithms such as FedProx \cite{qu2021federated}, Scaffold \cite{karimireddy2020scaffold}, and FedADMM \cite{wang2022fedadmm, zhou2023federated} have been proposed to completely remove the influence of unbounded data heterogeneity. A recent work \cite{cheng2023momentum} exploits momentum to address data heterogeneity and accelerates the convergence rate. 

In contrast, the PFL problem with separate shared and personal variables is relatively underexplored.  Several effective algorithms including FedRep \cite{collins2021exploiting}, FedSim \cite{pillutla2022federated}, and FedAlt \cite{pillutla2022federated} have been discussed in the introduction. Additionally, Wei et al. \cite{wei2023personalized} proposed a meta-learning based framework and analyzed its convergence behavior. Moreover, reference \cite{el2022personalized} developed an algorithm with a fast and theoretically tractable communication compression mechanism for PFL.

\subsection{Organization}
The remaining sections of this paper are structured as follows. Section \ref{sec-preliminary} introduces the assumptions and notations that will be used in the subsequent convergence analysis. The main results of our paper, including the FedAvg-P and Scaffold-P algorithms as well as their convergence analyses are presented in Sections \ref{sec-fedavg} and \ref{sec-scaffold}, respectively. In Section \ref{sec-experiment}, we provide numerical results to support and validate our findings.

\section{Preliminaries}\label{sec-preliminary}
In this section, we present the notations, assumptions, and  related preliminaries required for our theoretical results. 

\subsection{Notations and related models} \label{subsec-models}
The symbols we use throughout the paper 
are listed in Table \ref{tab-notation}.
Furthermore, we impose the condition $\gamma_u\eta_u = \gamma_v\eta_v$ in all convergence analyses and denote their value as $\gamma$, i.e., we let $\gamma = \gamma_u\eta_u = \gamma_v\eta_v$.
\begin{table}[!htbp]
\small 
\renewcommand\arraystretch{1}
\caption{Summary of notations used in this paper.}
\label{tab-notation}
\centering
\scalebox{1}{
\begin{tabular}{ p{1.8cm} | p{10cm}}
 \toprule 
 $n$, $m$, and $i$ & total number of clients, sampled number of clients, index of clients\\
 $[n]$ &client set $\{1,2,...,n\}$\\
 $T$, $t$ & number of outer loops, index of outer loops\\
 $K$, $k$ & number of local update steps, index of local update steps\\
 $u^t$, $\vv^t$, $v_i^t$ & shared and personal variables at the beginning of loop $t$ \\
 $u_{i,k}^t$, $v_{i,k}^t$ & shared and personal variables of the $i$th client in loop $t$ and step $k$\\
 $\gamma_u, \eta_u$ & inner and outer step sizes for shared variables $u$\\
 $\gamma_v, \eta_v$ & inner and outer step sizes for personal variables $v_i$\\
 $\cG_u^t, \cG_{\vv}^t, \hat \cG_{\vv}^t$ &expected gradient norms $\EE\|\nabla_u f(u^t,\vv^t)\|^2$, $\frac{1}{n}\sum_{i=1}^n\EE\|\nabla_{v_i} f_i(u^t,v_i^t)\|^2$, and $\frac{m}{n^2}\sum_{i=1}^n\EE\|\nabla_{v_i} f_i(u^t,v_i^t)\|^2$ \\
 $F_0$ &initial function value gap $f(u^0, \vv^0) - \inf_{u,\vv} f(u, \vv)$ \\
 \bottomrule
\end{tabular}
}
\end{table}

As discussed in the introduction, the partial model personalization problem \eqref{prob-general} serves as a general formulation that encompasses various important scenarios. For instance, it 
reduces to the traditional fully shared FL problem with a single common  variable if all personal variables $v_i$ are dropped:
\begin{align}\label{prob-full}
\min_{u} \quad f(u) := \frac{1}{n}\sum_{i=1}^n f_i(u),\quad \mbox{where} \quad f_i(u) = \EE_{\xi_i \sim \mathcal{D}_i}[F(u; \xi_i)].
\end{align}
Additionally, it reduces to the fully PFL problem when dropping shared variable $u$:
\begin{align}\label{prob-full-partial}
\min\limits_{\vv} \quad f(\vv) := \frac{1}{n}\sum_{i=1}^n f_i(v_i),\quad \mbox{where} \quad f_i(v_i) = \EE_{\xi_i \sim \mathcal{D}_i}[F(v_i; \xi_i)].
\end{align}

\subsection{Assumptions}\label{ass}
We make the following assumptions in our analysis. Assumptions \ref{ass-smooth} and \ref{ass-sto} are standard in  first-order stochastic algorithms, while Assumption \ref{ass-hete} constrains the data heterogeneity by bounding the gradient dissimilarity between different clients which is commonly used in \cite{karimireddy2020scaffold,yang2021achieving,pillutla2022federated}. 
\begin{assumption} [\sc smoothness] \label{ass-smooth}
For each $i\in[n]$, the function $f_i(u,v_i)$ is differentiable with respect to both $u$ and $v_i$.  Furthermore, we assume that	
\begin{itemize}[leftmargin=20pt]
		\item $\nabla_u f_i(u,v_i)$ is $L_u$-Lipschitz continuous in terms of $u$, and $L_{uv}$-Lipschitz continuous in terms of $v$.
		\item $\nabla_v f_i(u,v_i)$ is $L_v$-Lipschitz continuous in terms of $v_i$, and $L_{vu}$-Lipschitz continuous in terms of $u$.
\end{itemize}
We further denote $L = \max\{L_u, L_v, L_{uv}, L_{vu}\}$. 
\end{assumption}

\begin{assumption}[\sc Stochastic gradient] \label{ass-sto}
    For each $i\in[n]$, the stochastic gradient is unbiased for any $u \in \RR^{d_u}$ and $\vv \in \RR^{nd_v}$, i.e.,
    $\mathbb{E}[\nabla_u F(u,v_i;\xi_i)] = \nabla_u f_i(u,v_i)$, $
      \mathbb{E}[\nabla_v F(u,v_i;\xi_i)] = \nabla_v f_i(u,v_i) $,
    and the gradient variance is bounded, i.e., 
    \begin{align*}
        \mathbb{E}\|\nabla_u F(u,v_i;\xi_i) - \nabla_u f_i(u,v_i)\|^2  \le \sigma_u^2, \quad 
        \mathbb{E}\|\nabla_v F(u,v_i;\xi_i) - \nabla_v f_i(u,v_i)\|^2 \le \sigma_v^2.
    \end{align*} 
 Moreover, we assume each $\xi_i$ is independent of each other for any $i\in[n]$.
\end{assumption}
\begin{assumption}[\sc Bounded Gradient Dissimilarity, BGD]\label{ass-hete}
There exists a constant ${b} > 0$ such that for any $u \in \RR^{d_u}$ and $\vv \in \RR^{nd_v}$ 
 \begin{align}
	 \frac{1}{n}\sum_{i=1}^n \|\du f_i(u,v_i)\|^2 \le b^2 + \| \du f(u,\vv)\|^2.
\end{align}
\end{assumption}

\section{FedAvg-P Algorithm} \label{sec-fedavg} 
In this section, we present the personalized FedAvg (FedAvg-P) algorithm to solve problem  \eqref{prob-general} and provide its theoretical analysis.

\subsection{Algorithm development} The FedAvg-P algorithm is listed in Algorithm \ref{alg-fedavg}. FedAvg-P supports partial client participation, where only a subset $\mathcal{S}^t$ of clients is sampled to update their variables per outer loop (line 2). When $\mathcal{S}^t = [n]$, all clients participate in FedAvg-P updates. Each client $i$ maintains its personal variable $v_i$ and a local copy $u_i$ of the shared variable. In outer loop $t$ of FedAvg-P, each sampled client $i\in\mathcal{S}^t$ first retrieves the shared variable $u^t$ from the server (line 4), performs $K$ consecutive stochastic gradient descent steps to update $u_i$ and $v_i$ simultaneously (lines 6 and 7), and transmits the updated $u_{i,K}$ back so that the server can merge all local copies to obtain a new global shared variable $u^{t+1}$ (line 11). Compared to the standard FedAvg algorithm, FedAvg-P introduces additional updates to the personal variable $v$. Since $v$ is not shared, there is no need to communicate $v$ between clients.

\begin{algorithm}[thbp]
 \hspace*{\algorithmicindent} \small{\textbf{Input:}\ initialization $u^0, \vv^0$}. \\
 \hspace*{\algorithmicindent} 
 \small{\textbf{Output:}\ solution $u^T, \vv^T$.}
\begin{algorithmic}[1]
\For{each round $t=0,1,...,T-1$}
\vspace{1mm}
\State \small{Sample clients $\cS^t$ uniformly randomly so that $|\cS^t| = m$}. \Comment{\footnotesize{Partial Client {Participation}}}
\For{\small{all clients $i \in \cS^t$ \textbf{in parallel}}} 
\State Initialize $u_{i,0}^t \leftarrow u^t$, $v_{i,0}^t \leftarrow v_i^t$. \Comment{\footnotesize Communication}
\For{$k=0,1,...,K-1$}
\vspace{1mm}
\State $u_{i,k+1}^t \leftarrow u_{i,k}^t-\gamma_u\nabla_u F(u_{i,k}^t, v_{i,k}^t, \xi_{i,k}^t)$. \Comment{\footnotesize Local updates for $u$}
\State $v_{i,k+1}^t \leftarrow v_{i,k}^t-\gamma_v\nabla_v F(u_{i,k}^t, v_{i,k}^t, \xi_{i,k}^t)$. \Comment{\footnotesize Local updates for $\vv$}
\EndFor
\State $u_i^{t+1} \leftarrow u_{i,K}^t$, $v_i^{t+1} \leftarrow (1-\eta_v)v_i^t + \eta_v v_{i,K}^t$.
\EndFor
\State $u^{t+1} \leftarrow (1-\eta_u)u^t + \frac{\eta_u}{m}\sumpit u_i^{t+1}$. \Comment{\footnotesize Communication}
\EndFor
\caption{FedAvg-P Algorithm}\label{alg-fedavg}
\end{algorithmic}
\end{algorithm}

FedAvg-P can be tailored to match existing algorithms by adjusting specific parameters. For instance, eliminating the partial variables $v_i$ transforms the proposed algorithm into the FedAvg algorithm for traditional FL problem \eqref{prob-full} with a single shared variable. Conversely, removing the shared variable $u$ {and setting $K=1, m=n$} turns the algorithm into parallel stochastic gradient descent (SGD) for solving the fully personalized FL problem \eqref{prob-full-partial}. Notably, setting the outer step size parameters $\eta_u$ and $\eta_v$ to 1 in Algorithm \ref{alg-fedavg} aligns it with the FedSim algorithm \cite{pillutla2022federated}. These observations demonstrate the flexibility of FedAvg-P. Therefore, a convergence analysis for FedAvg-P is applicable to various existing approaches and problem formulations. 

\subsection{Convergence analysis with partial client participation}\label{sec-fedavg-result} 
This subsection will establish the convergence properties of FedAvg-P when {only} a subset of clients updates their variables per outer loop. The subsequent subsection will demonstrate improved convergence rates when all $n$ clients participate in the updates.

\subsubsection{Supporting lemmas} To accommodate partial client participation, we introduce virtual sequences $\hat u_i^t$ and $\hat v_i^t$ to simplify the analysis. For all $t,k$ and $i \in [n]$, we define  
\begin{align*}
	\hat u_{i,k+1}^t := \hat u_{i,k}^t - \gamma_u \nabla_u F(\hat u_{i,k}^t, \hat v_{i,k}^t, \xi_{i,k}^t) \text{ and }
	\hat v_{i,k+1}^t := \hat v_{i,k}^t - \gamma_v \nabla_v F(\hat u_{i,k}^t, \hat v_{i,k}^t, \xi_{i,k}^t), 
\end{align*}
where $\hat u_{i,0}^t$ and $\hat v_{i,0}^t$ are initialized as $u^t$ and $v_i^t$, respectively, at the beginning of each outer loop. In other words, we assume that the virtual sequences of all clients $i$ experience local updates, regardless of whether $i \in \cS^t$ or not. The equality $\hat u_{i,k}^t = u_{i,k}^t$ and $\hat v_{i,k}^t = v_{i,k}^t$ hold for $i\in \cS^t$ but not for $i \not\in \cS^t$. To facilitate convergence analysis, we  further introduce the following notations:
\begin{equation}\label{eqa-T}
	\begin{aligned}
		T_{1,u} &:= -\frac{\gamma}{n}\sumik\EE\langle \nabla_u f(u^t, \vv^t), \du f_i(\hat u_{i,k}^t, \hat v_{i,k}^t)\rangle,\\
		T_{1,\vv} &:= -\frac{\gamma m}{n^2}\sumik\EE\langle \nabla_v f_i(u^t, v_i^t), \dv f_i(\hat u_{i,k}^t, \hat v_{i,k}^t)\rangle, \\
		T_{2,u} &:= \frac{\gamma^2}{n^2}\EE\bigg\|\sumik\du f_i(\hat u_{i,k}^t, \hat v_{i,k}^t)\bigg\|^2\hspace{-1.5mm},\ T_{2,\vv} := \frac{\gamma^2m}{n^2}\sumi\EE\bigg\|\sumk\dv f_i(\hat u_{i,k}^t, \hat v_{i,k}^t)\bigg\|^2\hspace{-1.5mm},\\
		T_{3,u} &:= \sumi \EE\bigg\|\sumk \bigg(\du f_i(\hat u_{i,k}^t, \hat v_{i,k}^t) - \frac{1}{n}\sum_{j=1}^n \du f_j(\hat u_{j,k}^t, \hat v_{j,k}^t) \bigg)\bigg\|^2.
	\end{aligned}
\end{equation}
where variables $u, \vv$ and constants $m$ and $n$ are defined in Table \ref{tab-notation}, and {$\gamma = \gamma_u\eta_u = \gamma_v\eta_v$}.

The following lemma bounds the change in function value across consecutive outer loops for FedAvg-P with partial client participation. The proof can be found in Appendix \ref{sec-proof}.
\begin{lemma}[\sc Bounding difference in function value]
\label{lem-fedavg-step}
	Under Assumptions \ref{ass-smooth} and \ref{ass-sto}, it holds that 
	\begin{equation}\label{eqa-fedavg-step}
		\begin{aligned}
			\EE f(u^{t+1},\vv^{t+1}) - \EE f(u^t, \vv^t)
			&\le T_{1,u} + T_{1,\vv} + 2LT_{2,u} + 2LT_{2,\vv}\\
			&\quad + \frac{4L\gamma^2(n-m)}{mn^2}T_{3,u} + \frac{2LK\sigma_u^2\gamma^2}{m} +  \frac{2LK\sigma_v^2\gamma^2m}{n}.
		\end{aligned}
	\end{equation}
	Importantly, if $m=n$, the term containing $T_{3,u}$ vanishes.
\end{lemma}

Next we provide a series of lemmas to bound terms $T_{1,u}$, $T_{1,\vv}$, $T_{2,u}$, $T_{2,\vv}$ and $T_{3,u}$ respectively. To this end, we need to introduce $\cE^t := \sum_{i,k} \EE(\|\hat u_{i,k}^t-u^t\|^2 + \|\hat v_{i,k}^t - v_i^t\|^2)$ to quantify the deviation resulted from the local updates, where $\sum_{i,k}$ denotes the summation over all values of $i$ in the set $[n]$ and all values of $k$ in the set $\{0,1,...,K-1\}$. The following lemma provides an upper bound on $T_{1,u}$ and $T_{1,\vv}$.

\begin{lemma}[\sc bounding $T_{1,u}, T_{1,\vv}$]\label{lem-T1uv}
	Under Assumptions \ref{ass-smooth}, it holds that 
	\begin{align*}
		T_{1,u} \le -\frac{K\gamma}{2}\cG_u^t + \frac{L^2\gamma}{n}\cE^t \quad\text{and}\quad T_{1,\vv} \le -\frac{K\gamma m}{2n}\cG_{\vv}^t + \frac{L^2\gamma m}{n^2}\cE^t.
	\end{align*}
 where $\cG_u^t$ and $\cG_{\vv}^t$ are defined in Table \ref{tab-notation}. 
\end{lemma}
\begin{proof}
	We only prove the upper bound of $T_{1,u}$. A similar approach can be employed to obtain an upper bound for $T_{1, \vv}$. 
	First notice 
	\begin{align*}
		T_{1,u} 
		&= -\frac{\gamma}{n}\sumik\EE  \langle \nabla_u f(u^t,\vv^t), \du f_i(\hat u_{i,k}^t,\hat v_{i,k}^t) -\du f(u^t,\vv^t)+\du f(u^t,\vv^t)\rangle\\
		&= -K\gamma\EE\|\du f(u^t,\vv^t)\|^2 -  \frac{\gamma}{n}\sumik\EE \langle \nabla_u f(u^t,\vv^t), \du f_i(\hat u_{i,k}^t,\hspace{-0.7mm}\hat v_{i,k}^t) - \du f_i(u^t, v_i^t)\rangle, 
	\end{align*}
	where the final equation holds since $\du f(u^t, v_i^t) = \avgi\du f_i(u^t, v_i^t)$. Then we have
	\begin{align*}
		T_{1,u}	&\le -\frac{K\gamma}{2}\EE\|\du f(u^t,\vv^t)\|^2  + \frac{\gamma}{2n}\sumik\EE\|\du f_i(\hat u_{i,k}^t,\hat v_{i,k}^t)-\du f_i(u^t, v_i^t)\|^2\\
		&\le -\frac{K\gamma}{2}\EE\|\du f(u^t,\vv^t)\|^2 + \frac{L^2\gamma}{n}\sumik\EE(\|\hat u_{i,k}^t-u^t\|^2 + \|\hat v_{i,k}^t-v^t\|^2), 
	\end{align*}
	where the first inequality is due to $\langle a,b \rangle \le \frac{1}{2}\|a\|^2 + \frac{1}{2}\|b\|^2$, while the second is a result of  Assumption \ref{ass-smooth} and the mean inequality.
\end{proof}

Similarly,  the following lemma will bound $T_{2,u}$ and $T_{2,\vv}$.
\begin{lemma}[\sc bounding $T_{2,u}, T_{2, \vv}$]\label{lem-T2uv}
	Under Assumptions \ref{ass-smooth}, it holds that 
	\begin{align*}
		T_{2,u} \le 2K^2\gamma^2\cG_u^t + \frac{4L^2K\gamma^2}{n}\cE^t \quad\text{and}\quad T_{2,\vv} \le \frac{2K^2\gamma^2m}{n}\cG_{\vv}^t + \frac{4L^2K\gamma^2m}{n^2}\cE^t.
	\end{align*}
\end{lemma}
\begin{proof}
	The upper bound for $T_{2,u}$ can be proved using the mean inequality and Assumption \ref{ass-smooth}: 
	\begin{align*}
		T_{2,u}
		&= \frac{\gamma^2}{n^2}\EE\Big\|\sumik\Big(\du f_i(\hat u_{i,k}^t, \hat v_{i,k}^t)- \du f(u^t,\vv^t) + \du f(u^t,\vv^t)\Big)\Big\|^2\\
		&\le 2K^2\gamma^2\EE\|\du f(u^t,\vv^t)\|^2 + \frac{2K\gamma^2}{n}  \sumik \EE\left\|\du f_i(\hat u_{i,k}^t, \hat v_{i,k}^t) - \du f_i(u^t,v_i^t)\right\|^2\\
		&\le 2K^2\gamma^2\EE\|\du f(u^t,\vv^t)\|^2 + \frac{4L^2K\gamma^2}{n} \sumik\EE\left(\|\hat u_{i,k}^t-u^t\|^2 + \|\hat v_{i,k}^t-v_i^t\|^2\right).
	\end{align*}
	The proof for $T_{2,\vv}$ can be conducted in a similar manner.
\end{proof}

As mentioned in Lemma \ref{lem-fedavg-step}, we have $T_{3,u} = 0$ with full client participation. We only need to bound $T_{3,u}$ in the partial client participation scenario. 
\begin{lemma}[\sc bounding $T_{3,u}$]\label{lem-T3u}
	Under Assumptions \ref{ass-smooth} and \ref{ass-hete}, it holds that
	\begin{align*}
		T_{3,u} \le 3nK^2b^2+ 12L^2K\cE^t.
	\end{align*}
\end{lemma}
\begin{proof} Using the mean inequality, we obtain
	\begin{align*}
		T_{3,u} 
		&\le 3\sumi \EE\bigg\|\sumk\Big(\du f_i(\hat u_{i,k}^t, \hat v_{i,k}^t) - \du f_i(u^t, v_i^t)\Big)\bigg\|^2\\
		&\quad + 3\sumi \EE\bigg\|\frac{1}{n}\sum_{j=1}^n \sumk\Big( \du f_j(u^t, v_j^t) - \du f_j(\hat u_{j,k}^t, \hat v_{j,k}^t) \Big)\bigg\|^2\\
		&\quad + 3\sumi \EE\bigg\|\sumk\Big( \du f_i(u^t, v_i^t)- \du f(u^t, \vv^t)\Big) \bigg\|^2\\
		&\le 6K\sumik \EE\left\|\du f_i(\hat u_{i,k}^t, \hat v_{i,k}^t) - \du f_i(u^t, v_i^t)\right\|^2\\
		&\quad + 3K^2\sumi \EE\left\|\du f_i(u^t, v_i^t) - \du f(u^t, \vv^t)\right\|^2 \\
		&\le 12L^2K\sumik \EE\left(\|\hat u_{i,k}^t-u^t\|^2 + \|\hat v_{i,k}^t-v^t\|^2\right)+ 3nK^2b^2,
	\end{align*}
	where the final inequality derives from Assumptions \ref{ass-smooth} and \ref{ass-hete}.
\end{proof}

It is observed that the deviation term $\cE^t$ appears in Lemmas \ref{lem-T1uv}--\ref{lem-T3u}. The following lemma provides an upper bound for $\cE^t$.  To this end, we introduce another quantity $\cB^t := \sum_{i=1}^n \mathbb{E} \|\nabla_u f_i(u^t, v_i^t)\|^2$ whose upper bound will be established afterwards.
\begin{lemma}[\sc bounding $\cE^t$]\label{lem-fedavg-e} Under Assumptions \ref{ass-smooth} and \ref{ass-sto}, by setting the step sizes $\gamma_u$ and $\gamma_v$ such that $4L^2K(K-1)(\gamma_u^2+\gamma_v^2) \le 1$, we have
	\begin{align*}
		\cE^t &\le 8K^2(K-1)\gamma_u^2\cB^t + 8nK^3\gamma_v^2\cG_{\vv}^t + 4nK^2\left(\sigma_u^2\gamma_u^2+\sigma_v^2\gamma_v^2\right).
	\end{align*}
\end{lemma}

\begin{proof}
	For the case of $K=1$, the proof is straightforward. Now, let us consider the scenario where $K \ge 2$. For $k=0,1,...,K-1$, the following inequalities hold:
	\begin{align*}
		\EE\|\hat u_{i,k+1}^t-u^t\|^2 
		&= \EE\|\hat u_{i,k}^t-\gamma_u\du F(\hat u_{i,k}^t, \hat v_{i,k}^t, \xi_{i,k}^t)- u^t\|^2\\
		&\le \EE\|\hat u_{i,k}^t-\gamma_u\du f_i(\hat u_{i,k}^t, \hat v_{i,k}^t)- u^t\|^2 + \sigma_u^2\gamma_u^2 \\
		&\le \Big(1+\frac{1}{K-1}\Big)\EE\|\hat u_{i,k}^t-u^t\|^2 +K\gamma_u^2\EE\|\du f_i(\hat u_{i,k}^t, \hat v_{i,k}^t)\|^2+\sigma_u^2\gamma_u^2\\
		&\le \Big(1+\frac{1}{K-1}\Big)\EE\|\hat u_{i,k}^t-u^t\|^2  +2K\gamma_u^2\EE\|\du f_i(u^t, v_i^t)\|^2\\
		&\qquad +2K\gamma_u^2\EE\|\du f_i(\hat u_{i,k}^t, \hat v_{i,k}^t) - \du f_i(u^t, v_i^t)\|^2+\sigma_u^2\gamma_u^2 \\
		&\le \Big(1+\frac{1}{K-1}+4L^2K\gamma_u^2\Big)\EE\|\hat u_{i,k}^t-u^t\|^2 + 4L^2K\gamma_u^2 \EE\|\hat v_{i,k}^t-v_i^t\|^2\\
		&\quad +2K\gamma_u^2\EE\|\du f_i(u^t, v_i^t)\|^2+\sigma_u^2\gamma_u^2. 
	\end{align*}
	In the proof above, we first separate the mean and variance and use the triangle inequality. The second inequality follows from the inequality $2\langle a,b \rangle \le \kappa\|a\|^2+\frac{1}{\kappa}\|b\|^2$ and setting $\kappa = K-1$. The final inequality derives from Assumption \ref{ass-smooth}. 
	By employing a similar approach, we can derive the following inequality.
	\begin{align*}
		\EE\|\hat v_{i,k+1}^t-v_i^t\|^2 
		&\le \Big(1+\frac{1}{K-1}+4L^2K\gamma_v^2\Big)\EE\|\hat v_{i,k}^t-v_i^t\|^2 + 4L^2K\gamma_v^2 \EE\|\hat u_{i,k}^t-u_i^t\|^2\\ 
		&\quad+2K\gamma_v^2\EE\|\dv f_i(u^t, v_i^t)\|^2+\sigma_v^2\gamma_v^2.
	\end{align*}
	Let $\cE_{i,k}^t := \EE(\|\hat u_{i,k}^t-u^t\|^2+\|\hat v_{i,k}^t-v_i^t\|^2)$. The above two inequalities can be combined into 
	\begin{equation}\label{eqa-fedavg-ce}
		\begin{aligned}
			\cE_{i,k+1}^t &\le \Big(1+\frac{1}{K-1} + 4L^2K(\gamma_u^2+\gamma_v^2)\Big)\cE_{i,k}^t + 2K\gamma_u^2\EE\|\du f_i(u^t, v_i^t)\|^2\\
			&\qquad +2K\gamma_v^2\EE\|\dv f_i(u^t, v_i^t)\|^2+\sigma_u^2\gamma_u^2 +\sigma_v^2\gamma_v^2.
		\end{aligned}
	\end{equation}
	Considering $4L^2K(\gamma_u^2+\gamma_v^2) \le \frac{1}{K-1}$, apply \eqref{eqa-fedavg-ce} $k$ times to obtain
	\begin{align*}
		\cE_{i,k}^t &\le \left(\begin{array}{r} 
			2K\gamma_u^2\EE\|\du f_i(u^t, v_i^t)\|^2 + \sigma_u^2\gamma_u^2 \\
			+ 2K\gamma_v^2\EE\|\dv f_i(u^t, v_i^t)\|^2 + \sigma_v^2\gamma_v^2
		\end{array} \right) \sum\limits_{r=0}^{k-1}\left(1+\frac{2}{K-1}\right)^r\\
		&\le 8K(K-1)(\gamma_u^2\EE\|\du f_i(u^t, v_i^t)\|^2\!+\gamma_v^2\EE\|\dv f_i(u^t, v_i^t)\|^2)+4K(\sigma_u^2\gamma_u^2+\sigma_v^2\gamma_v^2).
	\end{align*}
	We employ induction to prove the first inequality. The second inequality derives from $\sum_{r=0}^{k-1}\big(1+\frac{2}{K-1}\big)^r \le \frac{K-1}{2}\big(1+\frac{2}{K-1}\big)^{k} \le \frac{K-1}{2}\big(1+\frac{2}{K-1}\big)^{K-1}$ and $1+x \le e^x$.
	Note that $\cE^t = \sumik \cE_{i,k}^t$, we have
	\begin{align}
		\cE^t &\le 8K^2(K-1)\! \sumi(\gamma_u^2\EE\|\du f_i(u^t, v_i^t)\|^2\!+\! \gamma_v^2\EE\|\dv f_i(u^t, v_i^t)\|^2) +4nK^2(\sigma_u^2\gamma_u^2\!+\!\sigma_v^2\gamma_v^2)  \nonumber \\
		&\le 8K^2(K-1)\gamma_u^2\cB^t+8nK^2(K-1)\gamma_v^2\cG_{\vv}^t+4nK^2(\sigma_u^2\gamma_u^2+\sigma_v^2\gamma_v^2). \nonumber 
	\end{align}
\end{proof}

Finally, we prove that $\cB^t$ increases sufficiently slow with proper step sizes in the following lemma, whose proof can be found in Appendix \ref{sec-proof}. 

\begin{lemma}[\sc bounding $\cB^t$]\label{lem-fedavg-b}
    The bound $\cB^t \le n\cG_u^t + nb^2$ holds under Assumption~\ref{ass-hete}. 
    Furthermore, with full client participation, $\cB^t$ can be bounded as  
    \begin{align*}
        \cB^t &\le \Big(1+\frac{2}{T-1}\Big)\cB^{t-1} + 2nT\Big(\cG_u^{t-1}+\cG_{\vv}^{t-1}\Big) + \frac{2nT}{K}\Big(\frac{\sigma_u^2}{n}+\sigma_v^2\Big)  
    \end{align*}
    without Assumption~\ref{ass-hete}, as long as Assumptions \ref{ass-smooth} and \ref{ass-sto} hold, and the step sizes satisfy $\eta_u \ge \max(T, \sqrt{n}), \eta_v \ge 1$, and $\gamma \le \frac{1}{8LK}$. 
\end{lemma}

\subsubsection{Convergence rate}  With the above lemmas, it is straightforward to establish the convergence property of FedAvg-P with partial client participation. 
\begin{theorem}[\sc Convergence with Partial Participation]
	\label{thm-fedavg}
    Under Assumptions \ref{ass-smooth}, \ref{ass-sto} and \ref{ass-hete}, if the step sizes satisfies $\eta_u \ge \sqrt{m}, \eta_v \ge \sqrt{\frac{n}{m}}$ and $\gamma \le \frac{1}{32LK}$, it holds that
    \begin{align*}
        \avgt(\cG_u^t + \hat \cG_{\vv}^t) &\le \frac{4F_0}{TK\gamma}\! +\! 12L\left(\frac{4(n-m)Kb^2}{mn}\!+\!\frac{\sigma_u^2}{m}\!+\! \frac{m\sigma_v^2}{n}\right)\gamma\! +\! \frac{128L^2K(K-1)b^2}{\eta_u^2}\gamma^2
    \end{align*}
     where $F_0$, $\cG_u^t$, and $\hat \cG_{\vv}^t$ are defined in Table \ref{tab-notation}.
\end{theorem}
\begin{proof}
	We start by substituting Lemmas \ref{lem-T1uv}, \ref{lem-T2uv}, and \ref{lem-T3u} into Lemma \ref{lem-fedavg-step}. By recalling the definition $\hat \cG_{\vv}^t = \frac{m}{n}\cG_{\vv}^t$, we obtain the following expression:
	\begin{equation}\label{eqa-fedavg-step1}
		\begin{aligned}
			\EE f(u^{t+1},\vv^{t+1}) &\le \EE f(u^t,\vv^t) + \Big(-\frac{K\gamma}{2} +4LK^2\gamma^2\Big)(\cG_u^t + \hat\cG_{\vv}^t)\\
			&\quad +\Big(\frac{L^2\gamma}{n}+\frac{56L^3K\gamma^2}{n}+\frac{L^2\gamma m}{n^2}+\frac{8L^3K\gamma^2m}{n^2}\Big)\cE^t\\
			&\quad + 12\Big(\frac{1}{m}-\frac{1}{n}\Big)LK^2b^2\gamma^2 + \frac{2LK\sigma_u^2\gamma^2}{m} + \frac{2LK\sigma_v^2\gamma^2m}{n}\\
			&\le \EE f(u^t,\vv^t)+ \Big(-\frac{K\gamma}{2}+4LK^2\gamma^2\Big)(\cG_u^t+\hat \cG_{\vv}^t) \\
			&\quad +\frac{4L^2\gamma}{n}\cE^t + 12\Big(\frac{1}{m}-\frac{1}{n}\Big)LK^2b^2\gamma^2 + \frac{2LK\sigma_u^2\gamma^2}{m} + \frac{2LK\sigma_v^2\gamma^2m}{n}, 
		\end{aligned}
	\end{equation}
	where the second inequality holds since $32LK\gamma \le 1$ and $m \le n$. We now apply Lemmas \ref{lem-fedavg-e} and \ref{lem-fedavg-b} in this case.
	\begin{align*}
		\cE^t 
		&\le 8K^2(K-1)\gamma_u^2\cB^t + 8nK^3\gamma_v^2\cG_{\vv}^t + 4nK^2(\sigma_u^2\gamma_u^2+\sigma_v^2\gamma_v^2)\\
		&\le \frac{8nK^2(K-1)b^2\gamma^2}{\eta_u^2} + 8nK^3\gamma^2\Big(\frac{1}{\eta_u^2}\cG_u^t + \frac{1}{\eta_v^2}\cG_{\vv}^t\Big) + 4nK^2\gamma^2\Big(\frac{\sigma_u^2}{\eta_u^2}+\frac{\sigma_v^2}{\eta_v^2}\Big)\\
		&\le \frac{8nK^2(K-1)b^2\gamma^2}{\eta_u^2} + \frac{nK^2\gamma}{4L}\Big(\cG_u^t+\hat \cG_{\vv}^t\Big) + \frac{nK\gamma}{4L}\Big(\frac{\sigma_u^2}{m}+\frac{m\sigma_v^2}{n}\Big),
	\end{align*}
	where the final inequality holds since $32LK\gamma \le 1, \eta_u \ge \sqrt{m}$, and $\eta_v \ge \sqrt{\frac{n}{m}}$. Together with \eqref{eqa-fedavg-step1}, we have
	\begin{align*}
		\EE f(u^{t+1},\vv^{t+1})
		&\le \EE f(u^t,\vv^t) + \Big(\!\!-\!\frac{K\gamma}{2}+5LK^2\gamma^2\Big) \left(\cG_u^t + \hat\cG_{\vv}^t\right) + \frac{32L^2K^2(K-1)b^2\gamma^3}{\eta_u^2} \\
		&\quad  + 12\Big(\frac{1}{m}-\frac{1}{n}\Big)LK^2b^2\gamma^2 + \frac{3LK\sigma_u^2\gamma^2}{m} + \frac{3LK\sigma_v^2\gamma^2m}{n}.
	\end{align*}
	Then rewrite the inequality as
	\begin{align*}
		\Big(\frac{1}{2} - 5 LK\gamma\Big)(\cG_u^t + \hat\cG_{\vv}^t)
		&\le \frac{\EE[f(u^t, \vv^t) - f(u^{t+1},\vv^{t+1})]}{K\gamma}  + \frac{32L^2K(K-1)b^2\gamma^2}{\eta_u^2}\\
		&\quad+ 3L\Big(\frac{4(n-m)Kb^2}{mn} + \frac{\sigma_u^2}{m} + \frac{m\sigma_v^2}{n}\Big)\gamma.
	\end{align*}
	The desired bound can then be obtained by taking average over $t=0,1,...,T-1$ and applying $5LK\gamma \le \frac{1}{4}$. 
\end{proof}

With proper step sizes, we have the following corollary. 
\begin{corollary}[\sc convergence rate]
	\label{cor-fedavg-rate}
	Under the same assumptions as in Theorem \ref{thm-fedavg}, by setting $\gamma = \big(
32LK\!+\!\sqrt{\frac{3LKT}{F_0} (\frac{4(n-m)Kb^2}{mn}\!+\!\frac{\sigma_u^2}{m}+\frac{m\sigma_v^2}{n})}\big)^{-1}$, $\eta_v \ge \sqrt{\frac{n}{m}}$, and $\eta_u \ge \max(\sqrt{m}, \sqrt{\frac{b^2T}{LF_0}})$,  
	it holds that
	\begin{align*}
		\avgt (\cG_u^t+\hat\cG_{\vv}^t)&\le 8\sqrt{3\left(\frac{4(n-m)Kb^2}{mn}+\frac{\sigma_u^2}{m} + \frac{m\sigma_v^2}{n}\right)\frac{LF_0}{KT}} + \frac{129LF_0}{T}.
	\end{align*}
\end{corollary}
\begin{remark}[\sc Understanding the Convergence Rate]
Corollary \ref{cor-fedavg-rate} provides explicit insights into how various key factors affect the convergence of FedAvg-P. Specifically, the convergence rate demonstrates that larger values of the gradient variances $\sigma_u^2$ and $\sigma_v^2$,  gradient dissimilarity $b^2$, and the gradient smoothness $L$ lead to slower convergence. It also shows that more local update steps $K$ or sampled clients $m$ accelerates the convergence, quantitatively justifying the benefit of exploiting more local updates and participating clients in the algorithm. In particular, when the value of $b^2$ is sufficiently small, the required number of outer loop iterations $T$ to achieve an $\epsilon$-accurate solution decreases linearly with the increase in $K$. This phenomenon is often referred to as achieving linear speedup with respect to $K$. A similar property also holds for $m$ when $\sigma_v^2$ is sufficiently small.
\end{remark}

\begin{remark}[\sc Faster rate than existing algorithms] 
Algorithms for solving problem \eqref{prob-general} have been proposed in several studies, and their convergence rates are listed in Table \ref{com-fedavg-partial}. It can be observed that our established rate in Corollary \ref{cor-fedavg-rate} for FedAvg-P is faster than existing results. 
\end{remark}
\begin{table}[!htbp]
\renewcommand\arraystretch{1.9}
\caption{Comparison of different algorithms to solve PFL problem \eqref{prob-general}.}
\label{com-fedavg-partial}
\centering
\scalebox{1}{
\begin{tabular}{ p{3.1cm}|p{8.8cm}}
 \toprule
   Algorithms & Convergence Rate \\
 \midrule
 FedSim \cite{pillutla2022federated} 
 &$\frac{1}{\sqrt{T}}\sqrt{\sigma_u^2 + \frac{m\sigma_v^2}{n} + \left(1-\frac{m}{n}\right)b^2} +  {\left(\frac{(b^2+\sigma_u^2+\sigma_v^2)(K-1)}{KT^2}\right)}^{\frac{1}{3}} + \frac{1}{T}$  \\
  FedAlt \cite{pillutla2022federated} 
 &$\frac{1}{\sqrt{T}}\sqrt{\sigma_u^2 + \frac{m\sigma_v^2}{n} + \left(1-\frac{m}{n}\right)b^2} +  {\left(\frac{(b^2+\sigma_u^2+\frac{m\sigma_v^2}{n})(K-1)}{KT^2}\right)}^{\frac{1}{3}} + \frac{1}{T}$  \\
 \textbf{FedAvg-P (Ours)}
 &
 $\boldsymbol{\frac{1}{ 
\sqrt{KT}}\sqrt{\frac{\sigma_u^2}{m} + \frac{m\sigma_v^2}{n} + (\frac{1}{m}-\frac{1}{n})Kb^2} + \frac{1}{T}}$
 \\
  \bottomrule
\end{tabular}
}
\end{table}
\subsubsection{Convergence rate in special scenarios} 
As discussed in Section \ref{subsec-models}, the partial model personalization problem \eqref{prob-general} provides a general formulation that encompasses various important scenarios. When all personal variables $v_i$ are removed, problem \eqref{prob-general} reduces to the classical FL problem \eqref{prob-full}. In this scenario, the convergence rate of FedAvg-P established in Corollary \ref{cor-fedavg-rate} reduces to the following one of FedAvg:
\begin{align*}
\avgt \EE\|\nabla f(u^t)\|^2 \le 8\sqrt{3\left(\frac{4(n-m)Kb^2}{mn}+\frac{\sigma_u^2}{m} \right)\frac{LF_0}{KT}} + \frac{129LF_0}{T}
\end{align*}
by removing the influence of $\sigma_v^2$ from the rate. This result is superior to {{existing best-known} convergence rates for FedAvg in the literature, see the upper part of Table~\ref{tab-fedavg-full} for detailed comparison. 

\begin{table}[!htbp]
\renewcommand\arraystretch{1.75}
\caption{Convergence rates of FedAvg when solving the full model \eqref{prob-full}.}
\label{tab-fedavg-full}
\centering
\label{tab}
\scalebox{1}{
\begin{threeparttable}
\begin{tabular}{l|l|c|c}
 \toprule
  Ref. & Convergence Rate & Assumption &Sampling \tnote{1} \\
 \hline
 \cite{karimireddy2020scaffold}& $\frac{1}{\sqrt{mKT}}\sqrt{\sigma_u^2+K\left(1-\frac{m}{n}\right)b^2} + (\frac{b^2}{T^2})^{\frac{1}{3}} + \frac{1}{T}$ &\centering{BGD}& Yes \\
 \cite{yang2021achieving}& $\frac{1}{\sqrt{mKT}}\left(\sigma_u^2+K\left(1-\frac{m}{n}\right)b^2\right) + \frac{Kb^2 + \sigma_u^2 + 1}{T}$ &\centering{BGD}& Yes \\
 \textbf{Ours}& $\boldsymbol{\frac{1}{\sqrt{mKT}}\sqrt{\sigma_u^2+K\left(1-\frac{m}{n}\right)b^2} + \frac{1}{T}}$  &\centering{\textbf{BGD}} &\textbf{Yes}\\
  \hline
   \cite{yu2019parallel}& $\sqrt{\frac{\sigma_u^2}{nKT}} + \frac{b^2nK}{T}$ &\centering{BG\tnote{2}} &No \\
 \cite{karimireddy2020scaffold}& $\sqrt{\frac{\sigma_u^2}{nKT}} + (\frac{b^2}{T^2})^{\frac{1}{3}} + \frac{1}{T}$ &\centering{BGD} &No \\
 \cite{yang2021achieving}& $\frac{\sigma_u^2}{\sqrt{nKT}} + \frac{Kb^2 + \sigma_u^2 + 1}{T}$ &\centering{BGD} &No \\
 \textbf{Ours}& $\boldsymbol{\sqrt{\frac{\sigma_u^2}{nKT}} + \frac{1}{T}}$  &  \centering{$\boldsymbol{-}$} & \textbf{No}\\
 \bottomrule
\end{tabular}
\begin{tablenotes}
\footnotesize
\item[1] {The column labeled ``Sampling" specifies whether the analysis applies to the partial client participation scenario. When ``no sampling" is indicated, it signifies that the analysis exclusively pertains to the scenario with full client participation, i.e., $m=n$.}
\vspace{1mm}
\item[2] The BG (Bounded Gradient) assumption refers to $\frac{1}{n}\sum_{i=1}^n \|\du f_i(u,v_i)\|^2 \le b^2$, which is much stronger than the BGD Assumption \ref{ass-hete}. 
\end{tablenotes}
\end{threeparttable}
}
\end{table}

On the other hand, problem \eqref{prob-general} reduces to the fully PFL problem \eqref{prob-full-partial} when the shared variable $u$ is removed from the problem formulation. In this scenario, the convergence rate of FedAvg-P reduces to  
\begin{align*}
	\frac{1}{nT}\sum_{i=1}^n \sum_{t=0}^{T-1}\EE\|\nabla f_i(v^t_i)\|^2 \le 8\sqrt{\frac{3 \sigma_v^2LF_0}{T}} + \frac{129LF_0}{T}
\end{align*}
by setting $K=1$, $m=n$ and $\sigma_u^2 = 0$. This result matches the state-of-the-art rate of the vanilla parallel SGD algorithm established in literature when solving model \eqref{prob-full-partial}.

\subsection{Convergence analysis with full client participation}\label{sec-fedavg-result-full} 
Theorem \ref{thm-fedavg} and Corollary \ref{cor-fedavg-rate} establish the convergence rate of FedAvg-P with partial client participation. This section will establish an enhanced convergence property of FedAvg-P when all clients participate in every outer loop update.

\subsubsection{Convergence rate} When all clients participate in every outer loop update, the following theorem establishes the convergence rate for FedAvg-P.
\begin{theorem}[\sc Convergence with full Participation]
\label{thm-fedavg-f}
Under Assumptions \ref{ass-smooth} and \ref{ass-sto}, if all clients participate in every outer loop ($m=n$) and the outer step size satisfies $\eta_u \ge \max(T, \sqrt{\cB^0T/nLF_0}, \sqrt{n}), \eta_v \ge 1$ and $\gamma \le \frac{1}{84LK}$, it holds that
\begin{align*}
     \avgt(\cG_u^t + \cG_{\vv}^t) &\le \frac{4F_0}{TK\gamma} + 44L\Big(\frac{\sigma_u^2}{n}+ \sigma_v^2\Big)\gamma + \frac{LF_0}{T}.
\end{align*}
Furthermore, when $\gamma = \left(84LK + \sqrt{\frac{11LKT}{F_0}(\frac{\sigma_u^2}{n}+\sigma_v^2)}\right)^{-1}$,
it holds that
\begin{align*}
\avgt(\cG_u^t+\cG_{\vv}^t)&\le 8\sqrt{\frac{11LF_0}{KT} \Big(\frac{\sigma_u^2}{n}+\sigma_v^2\Big)}+ \frac{337LF_0}{T}, 
\end{align*}
where $\cB^0 = \sum_i \EE\|\du f(u^0, v_i^0)\|^2$.
\end{theorem}
\begin{proof}
The first part of the proof is nearly identical to that of Theorem \ref{thm-fedavg}, with the only difference being that $m=n$ and $T_{3,u}=0$ in this case. By utilizing Lemmas \ref{lem-T1uv} and \ref{lem-T2uv} as well as considering $32LK\gamma \le 1$, we obtain the following inequality:
\begin{equation}\label{eqa-fedavg-step2}
\begin{aligned}
\EE f(u^{t+1},\vv^{t+1})
&\le \EE f(u^t,\vv^t) + \Big(-\frac{K\gamma}{2}+4LK^2\gamma^2\Big)(\cG_u^t + \cG_{\vv}^t)\\
&\quad +\frac{4L^2\gamma}{n}\cE^t + \frac{2LK\sigma_u^2\gamma^2}{n} + 2LK\sigma_v^2\gamma^2.
\end{aligned}
\end{equation}
Utilizing Lemma \ref{lem-fedavg-e} alongside the stipulated step-size conditions $32LK\gamma \le 1, \eta_u \ge \sqrt{n}, \eta_v \ge 1$, we can derive 
\begin{align*}
    \cE^t 
    &\le \frac{K^2\gamma}{4L\eta_u^2}\cB^t + \frac{nK^2\gamma}{4L}\cG_{\vv}^t + \frac{nK\gamma}{4L} \Big(\frac{\sigma_u^2}{n}+\sigma_v^2\Big),
    \end{align*}
Together with \eqref{eqa-fedavg-step2}, it holds that
\begin{align*}
\EE f(u^{t+1},\vv^{t+1})
&\le \EE f(u^t,\vv^t) + \Big(-\frac{K\gamma}{2}+5LK^2\gamma^2\Big)(\cG_u^t+\cG_{\vv}^t)\\
&\quad + \frac{LK^2\gamma^2}{n\eta_u^2}\cB^t + \frac{3LK\sigma_u^2\gamma^2}{n} + 3LK\sigma_v^2\gamma^2.
\end{align*}
Taking average over $t=0,1,...,T-1$, we obtain the following inequality:
\begin{equation}\label{eqa-fedavg-full1}
\begin{aligned}
    \Big(\frac{1}{2} - 5LK\gamma\Big)\avgt(\cG_u^t + \cG_{\vv}^t) &\le \frac{F_0}{TK\gamma} + 3L\Big(\frac{\sigma_u^2}{n}+ \sigma_v^2\Big)\gamma +\frac{LK\gamma}{nT\eta_u^2}\sum\limits_{t=0}^{T-1}\cB^t.
\end{aligned}
\end{equation}
Recalling Lemma \ref{lem-fedavg-b} and employing deductive reasoning yields that
\begin{align*}
\sum\limits_{t=0}^{T-1}\cB^t &\le \sum\limits_{t=0}^{T-1}\left(\Big(1+\frac{2}{T-1}\Big)^t\cB^0 + 2nT\sum\limits_{r=0}^{t-1}(\cG_u^r+\cG_{\vv}^r)\Big(1+\frac{2}{T-1}\Big)^{t-1-r}\right. \\
&\left.\quad + \frac{2nT}{K}\Big(\frac{\sigma_u^2}{n}+\sigma_v^2\Big)\sum\limits_{r=0}^{t-1}\Big(1+\frac{2}{T-1}\Big)^r\right)\\
&\le 8T\cB^0 + 16nT^2\sum\limits_{t=0}^{T-1}(\cG_u^t+\cG_{\vv}^t) + \frac{8nT^3}{K}\Big(\frac{\sigma_u^2}{n}+\sigma_v^2\Big),
\end{align*}
where the final step derives from the fact that $\sum_{r=0}^{t-1}\big(1+\frac{2}{T-1}\big)^r \le \frac{T-1}{2}\big(1+\frac{2}{T-1}\big)^t, \big(1+\frac{2}{T-1}\big)^t \le \big(1+\frac{2}{T-1}\big)^{T-1}$ and $1+x\le e^x$.
Combine this inequality with \eqref{eqa-fedavg-full1} and use $\eta_u \ge T, 32LK\gamma \le 1$ to obtain
\begin{align*}
     \Big(\frac{1}{2} - 21LK\gamma\Big)\avgt(\cG_u^t + \cG_{\vv}^t) &\le \frac{F_0}{TK\gamma} + 11L\Big(\frac{\sigma_u^2}{n}+ \sigma_v^2\Big)\gamma + \frac{\cB^0}{4n\eta_u^2}.
\end{align*}
By utilizing $21LK\gamma \le \frac{1}{4}$ and $\eta_u \ge \sqrt{\cB^0T/nLF_0}$, we complete the proof.
\end{proof}
\begin{remark}[\sc Relaxing gradient dissimilarity assumption] 
According to Theorem \ref{thm-fedavg-f}, FedAvg-P with full client participation can converge without requiring any bounded gradient dissimilarity assumptions. Thus, the gradient dissimilarity bound $b^2$
does not affect the rate. This contrasts with FedAvg-P with partial client participation, whose convergence rate becomes slower when a large gradient dissimilarity bound $b^2$ is present.
\end{remark}

\subsubsection{Convergence rate for problem \eqref{prob-full}} 
By setting $\sigma_v^2 = 0$, the convergence rate established in Theorem \ref{thm-fedavg-f} reduces to the rate of FedAvg for solving problem \eqref{prob-full} with full client participation
\begin{align*}
\avgt \EE\|\nabla f(u^t)\|^2 &\leq 8\sqrt{\frac{11\sigma_u^2LF_0}{nKT}} + \frac{337LF_0}{T}
\end{align*}
without assuming any bounded gradient dissimilarity. This is a novel result since existing analyses for FedAvg rely on this restrictive assumption. The lower part of Table \ref{tab-fedavg-full} lists several existing FedAvg convergence rates for problem \eqref{prob-full}. It is observed that our derived rate outperforms all baselines under the mildest assumptions. 

\section{Scaffold-P Algorithm}\label{sec-scaffold}
While FedAvg-P with full client participation can converge without the bounded gradient dissimilarity assumption, this assumption remains necessary under partial client participation. This section presents the personalized Scaffold (Scaffold-P) algorithm, which completely remove the influence of $b^2$ even with partial client participation.

\subsection{Algorithm development}
Scaffold-P adapts the vanilla Scaffold algorithm \cite{karimireddy2020scaffold} originally designed for FL problem \eqref{prob-general} to the personalized setting with additional per-client variables (see Algorithm \ref{alg-scaffold}). Scaffold-P shares the same overall structure as FedAvg-P. The key difference in Scaffold-P lies in the introduction of auxiliary control variables $c_i$ (line 10) and $c$ (line 13) that respectively track the individual gradient $\nabla_u f_i(u, v_i)$ within each client $i$ and the globally averaged gradient $\nabla_u f(u, v_i)$. Since each bias-corrected stochastic gradient used to update $u$ in line 6 will gradually approach the same globally averaged gradient, i.e.,
\begin{align*}
\nabla_u F(u_{i,k}^t, v_{i,k}^t, \xi_{i,k}^t) - c_i^t + c^t \approx \frac{1}{n}\sum_{i=1}^n \nabla_u f_i(u^t, v_i^t),
\end{align*}
the tracking and correction enabled by $c_i$ and $c$ resolves the deviation caused by 
heterogeneous local updates. This eliminates the detrimental impact of gradient dissimilarity across clients. For this reason, Scaffold-P can overcome gradient dissimilarity issue suffered in FedAvg-P. 

Similarly to FedAvg-P, Scaffold-P is a general algorithm that can reduce to existing algorithms 
through the manipulation of specific parameters. 
When the personal variables $v_i$ are omitted from the algorithm, Scaffold-P effectively reverts to the Scaffold algorithm designed for problem \eqref{prob-full} with a single shared variable. 
On the other hand, when {$K=1, m=n$} and the shared variable $u$ is excised from the algorithm, Scaffold-P transforms into parallel SGD for solving problem \eqref{prob-full-partial}. Therefore, a convergence analysis for Scaffold-P is applicable to various existing approaches and problem formulations. 

\begin{algorithm}[tbp]
 \hspace*{\algorithmicindent} \small{\textbf{Input:}\ initialization $u^0, \vv^0, c_i^0=\frac{1}{K}\sum_k\du F(u^0, v_i^0, \xi^{-1}_{i,k}), c^0=\frac{1}{n}\sum_i c_i^0$}.\\
 \hspace*{\algorithmicindent} \small{\textbf{Output:}\ solution $u^T, \vv^T$}.
\begin{algorithmic}[1]
\For{each round $t=0,1,...,T-1$} 
\vspace{1mm}
\State \small{Sample clients $\cS^t$ uniformly randomly so that $|\cS^t| = m$}. \Comment{\footnotesize Partial Client Participation}
\For{\small{all clients $i \in \cS^t$ \textbf{in parallel}}}
\vspace{1mm}
\State Initialize $u_{i,0}^t \leftarrow u^t$, $v_{i,0}^t \leftarrow v_i^t$. \Comment{\footnotesize Communication}
\For{\small{$k=0,1,...,K-1$}}
\vspace{1mm}
\State $u_{i,k+1}^t \leftarrow u_{i,k}^t-\gamma_u(\nabla_u F(u_{i,k}^t, v_{i,k}^t, \xi_{i,k}^t) - c_i^t + c^t)$. \Comment{\footnotesize Local Updates for $u$}
\State $v_{i,k+1}^t \leftarrow v_{i,k}^t-\gamma_v\nabla_v F(u_{i,k}^t, v_{i,k}^t, \xi_{i,k}^t)$. \Comment{\footnotesize Local Updates for $\vv$}
\vspace{1mm}
\EndFor
\vspace{1mm}
\State $u_i^{t+1} \leftarrow u_{i,K}^t$, $v_i^{t+1} \leftarrow (1-\eta_v)v_i^t + \eta_v v_{i,K}^t$.  
\State $c_i^{t+1} \leftarrow c_i^t - c^t + \frac{1}{K\gamma_u}(u^t - u_i^{t+1})$.
\EndFor
\vspace{1mm}
\State $u^{t+1} \leftarrow (1-\eta_u) u^t + \frac{\eta_u}{m}\sumpit u_i^{t+1}$. \Comment{\footnotesize Communication}
\State $c^{t+1} \leftarrow  c^t + \frac{1}{n}\sumpit (c_i^{t+1} - c_i^t)$. \Comment{\footnotesize Communication}
\EndFor
\caption{Scaffold-P Algorithm} \label{alg-scaffold}
\end{algorithmic}
\end{algorithm}

\subsection{Convergence analysis} This section provides convergence guarantees for Scaffold-P and clarifies its ability to overcome gradient dissimilarity across clients.  

\subsubsection{Supporting lemmas} Similar to the analysis of FedAvg-P, we first introduce the virtual sequence for all $t,k$ and $i \in [n]$ as follows: 
\begin{align*}
\hat u_{i,k+1}^t &= \hat u_{i,k}^t\! -\!\gamma_u (\nabla_u F(\hat u_{i,k}^t, \hat v_{i,k}^t, \xi_{i,k}^t)\! -\! c_i^t\! +\! c^t),\ \hat v_{i,k+1}^t = \hat v_{i,k}^t\! -\! \gamma_v \nabla_v F(\hat u_{i,k}^t, \hat v_{i,k}^t, \xi_{i,k}^t).
\end{align*}
At the beginning of each outer loop, we set $\hat u_{i,0}^t = u^t$ and $\hat v_{i,0}^t = v_i^t$. Therefore, the equality $\hat u_{i,k}^t = u_{i,k}^t$ and $\hat v_{i,k}^t = v_{i,k}^t$ still holds for $i\in \cS^t$ in this case. Moreover, the following sequences are defined to facilitate the analysis of Scaffold-P: 
\begin{align*}
&\alpha_{i,k}^t := \cht\hat u_{i,k}^t + (1-\cht)\alpha_{i,k}^{t-1},\ \beta_{i,k}^t := \cht\hat v_{i,k}^t + (1-\cht)\beta_{i,k}^{t-1},\\
&d_i^t := \avgk\du f_i(\alpha_{i,k}^{t-1}, \beta_{i,k}^{t-1}), \text{ and } \ d^t := \avgi d_i^t. 
\end{align*}
Furthermore, we initialize  $\alpha_{i,k}^{-1} = u^0 \text{ and } \beta_{i,k}^{-1} = v_i^0$.

We maintain the notations $T_{1,u}, T_{1,\vv}, T_{2,u}$ and $T_{2,\vv}$, as originally defined in equation \eqref{eqa-T}. Furthermore, we introduce an additional notation as follows:
\begin{equation}\label{eqa-S}
\begin{aligned}
S_{3,u} := \sumi\EE\bigg\|\sumk\bigg(\du f_i(\hat u_{i,k}^t, \hat v_{i,k}^t) - d_i^t + d^t- \avgj\du f_j(\hat u_{j,k}^t, \hat v_{j,k}^t)\bigg)\bigg\|^2.
\end{aligned}
\end{equation}

The following lemma provides an upper bound on the difference between the function values at two consecutive outer iterations for Scaffold-P.
\begin{lemma}[\sc Bounding difference in function value]
\label{lem-scaffold-step}
Under Assumptions \ref{ass-smooth} and \ref{ass-sto}, it holds that 
\begin{equation}\label{eqa-scaffold-step}
\begin{aligned}
\EE f(u^{t+1},\vv^{t+1}) - \EE f(u^t, \vv^t)
      &\le T_{1,u} + T_{1,\vv} + 2LT_{2,u} + 2LT_{2,\vv} \\
      &\quad + \frac{4L\gamma^2(n-m)}{mn^2}S_{3,u} + \frac{18LK\sigma_u^2\gamma^2}{m} +  \frac{2LK\sigma_v^2\gamma^2m}{n}.
\end{aligned}
\end{equation}
\end{lemma}
The detailed proof is provided in Appendix \ref{sec-proof}. We continue to use $\cE^t$ to quantify the deviation resulting from local updates. It is worth noting that Lemma \ref{lem-T1uv} and Lemma \ref{lem-T2uv} still hold in this case. Therefore, we only need to bound $S_{3,u}$. To this end, we introduce the term $\cC^t := \sum_{i,k}\EE (\|\alpha_{i,k}^{t-1}-u^t\|^2 + \|\beta_{i,k}^{t-1}-v_i^t\|^2)$.
\begin{lemma}[\sc bounding $S_{3,u}$]\label{lem-S3u}
Under Assumptions \ref{ass-smooth}, it holds that
\begin{align*}
		&\EE S_{3,u} \le 16L^2K\cC^t + 16L^2K\cE^t.
\end{align*}
\end{lemma}
\begin{proof}
By applying the mean inequality several times and finally Assumption \ref{ass-smooth}, we have
\begin{align*}
S_{3,u}
&\le \sumi2\EE\bigg\|\sumk\left(\du f_i(\hat u_{i,k}^t, \hat v_{i,k}^t) - d_i^t\right)\bigg\|^2 + \frac{2}{n}\EE\bigg\|\sum_{i,k}\left(\du f_i(\hat u_{i,k}^t, \hat v_{i,k}^t) - d^t\right)\bigg\|^2\\
&\le 4K\sumik\EE\left( \begin{array}{l}
    2\left\|\du f_i(u^t, \vv^t) - \du f_i(\alpha_{i,k}^{t-1}, \beta_{i,k}^{t-1})\right\|^2 \\
    + 2\left\|\du f_i(\hat u_{i,k}^t, \hat v_{i,k}^t) - \du f_i(u^t, \vv^t)\right\|^2 
\end{array} \right) \\
&\le 16L^2K\sumik \EE \left(\|\alpha_{i,k}^{t-1}-u^t\|^2 \!+\! \|\beta_{i,k}^{t-1}-v_i^t\|^2 + \|\hat u_{i,k}^{t}-u^t\|^2 \!+\! \|\hat v_{i,k}^{t}-v_i^t\|^2\right).
\end{align*}
\end{proof}

The upper bound on $\cE^t$ is established for Scaffold-P in Lemma~\ref{lem-scaffold-e}, whose proof is similar to that of Lemma \ref{lem-fedavg-e} for FedAvg-P, and can be found in Appendix~\ref{sec-proof}.
\begin{lemma}[\sc bounding $\cE^t$]\label{lem-scaffold-e}
Under Assumptions \ref{ass-smooth} and \ref{ass-sto}, by setting the step sizes $\gamma_u, \gamma_v$ such that $12L^2K(K-1)(\gamma_u^2 + \gamma_v^2) \le 1$, we have
\begin{align*}
    \cE^t &\le 96L^2K^2\gamma_u^2\cC^t + 24nK^3\gamma_u^2\cG_u^t + 8nK^3\gamma_v^2\cG_{\vv}^t +52nK^2(\sigma_u^2\gamma_u^2+\sigma_v^2\gamma_v^2).
\end{align*}
\end{lemma}

We have so far bounded all the terms on the right-hand side of \eqref{eqa-scaffold-step} utilizing $\cG_u^t$, $\cG_{\vv}^t$, and $\cC^t$. Now, we   show that $\cC^t$ exhibits a diminishing property that allows it to be bounded in Lemma~\ref{lem-scaffold-c}, whose proof can be found in Appendix \ref{sec-proof}.
\begin{lemma}[\sc bounding $\cC^t$]\label{lem-scaffold-c}
Under Assumptions \ref{ass-smooth} and \ref{ass-sto}, by setting the step sizes $\eta_u \ge \sqrt{m}, \eta_v \ge \sqrt{\frac{n}{m}}$, and $\gamma \le \frac{1}{72LK} \min(\frac{m}{n^{2/3}},1)$, we have
\begin{equation}\label{eqa-scaffold-c}
\begin{aligned}
    \cC^t &\le \Big(1-\frac{m}{4n}\Big)\cC^{t-1} + \frac{mn^{2/3}K}{36L^2}(\cG_u^{t-1} + \hat \cG_{\vv}^{t-1}) + \frac{m^2}{18L^2}\Big(\frac{\sigma_u^2}{m} + \frac{m\sigma_v^2}{n}\Big).
\end{aligned}
\end{equation}
\end{lemma}

\subsection{Convergence rate} With the above supporting lemmas, we are ready to establish the convergence rate of Scaffold-P with partial client participation. 
\begin{theorem}[\sc Convergence with partial participation]
\label{thm-scaffold}
Under Assumptions \ref{ass-smooth} and \ref{ass-sto}, if the global step sizes satisfy $\eta_u \ge \sqrt{m}, \eta_v \ge \sqrt{\frac{n}{m}}$ and $\gamma \le \frac{1}{72LK}\min(\frac{m}{n^{2/3}}, 1)$, it holds that
\begin{align*}
    \frac{1}{T}\sum\limits_{t=0}^{T-1}(\cG_u^t + \hat \cG_{\vv}^t) &\le \frac{4F_0}{TK\gamma} + 148\Big(\frac{\sigma_u^2}{m}+\frac{\sigma_v^2m}{n}\Big)L\gamma.
\end{align*}
\end{theorem}
\begin{proof} 
We begin with applying Lemmas \ref{lem-T1uv}, \ref{lem-T2uv}, and \ref{lem-S3u} to Lemma \ref{lem-scaffold-step} to obtain the following inequality:
\begin{equation}\label{eqa-scaffold-step1}
\begin{aligned}
\EE f(u^{t+1},\vv^{t+1}) 
&\le \EE f(u^t,\vv^t) + \Big(-\frac{K\gamma}{2}+4LK^2\gamma^2\Big)(\cG_u^t+\hat\cG_{\vv}^t)\\
&\quad +\Big(\frac{L^2\gamma}{n}+\frac{72L^3K\gamma^2}{n}+\frac{L^2\gamma m}{n^2}+\frac{8L^3K\gamma^2m}{n^2}\Big)\cE^t\\
&\quad + \frac{64L^3K\gamma^2}{mn}\cC^t + \frac{18LK\sigma_u^2\gamma^2}{m} + \frac{2LK\sigma_v^2\gamma^2m}{n}\\
&\le \EE f(u^t,\vv^t) + \Big(-\frac{K\gamma}{2}+4LK^2\gamma^2\Big)(\cG_u^t+\hat\cG_{\vv}^t)\\
&\quad +\frac{4L^2\gamma}{n}\cE^t + \frac{64L^3K\gamma^2}{mn}\cC^t + \frac{18LK\sigma_u^2\gamma^2}{m} + \frac{2LK\sigma_v^2\gamma^2m}{n}, 
\end{aligned}
\end{equation}
where the final inequality holds since $72LK\gamma \le 1$ and $m\le n$. By Lemma \ref{lem-scaffold-e}, we have
\begin{align*}
    \cE^t &\le 96L^2K^2\gamma_u^2\cC^t + 24nK^3\gamma_u^2\cG_u^t + 8nK^3\gamma_v^2\cG_{\vv}^t +52nK^2(\sigma_u^2\gamma_u^2+\sigma_v^2\gamma_v^2)\\
    &\le  \frac{96L^2K^2\gamma^2}{\eta_u^2}\cC^t + 24nK^3\gamma^2\Big(\frac{1}{\eta_u^2}\cG_u^t + \frac{1}{\eta_v^2}\cG_{\vv}^t\Big) +52nK^2\gamma^2\Big(\frac{\sigma_u^2}{\eta_u^2}+\frac{\sigma_v^2}{\eta_v^2}\Big)\\
    &\le \frac{2LK\gamma}{m}\cC^t + \frac{nK^2\gamma}{2L}\big(\cG_u^t + \hat \cG_{\vv}^t\big) + \frac{3nK\gamma}{4L}\Big(\frac{\sigma_u^2}{m}+\frac{m\sigma_v^2}{n}\Big),
\end{align*}
where the last inequality derives from $72LK\gamma \le 1, \eta_u \ge \sqrt{m}$, and $\eta_v \ge \sqrt{\frac{n}{m}}$. Together with \eqref{eqa-scaffold-step1}, we obtain
\begin{align*}
\EE f(u^{t+1},\vv^{t+1})
&\le \EE f(u^t,\vv^t) + \Big(-\frac{K\gamma}{2}+ 6LK^2\gamma^2\Big)(\cG_u^t + \hat \cG_{\vv}^t) \\
&\quad + \frac{72L^3K\gamma^2}{mn}\cC^t + \frac{21LK\sigma_u^2\gamma^2}{m} + \frac{21LK\sigma_v^2\gamma^2m}{n}.
\end{align*}
This inequality can be rewritten as
\begin{align*}
    \Big(\frac{1}{2}-6LK\gamma\Big)(\cG_u^t + \hat\cG_{\vv}^t) 
    \le{}& \frac{\EE f(u^t, \vv^t) - \EE f(u^{t+1},\vv^{t+1})}{K\gamma}\\
    & + 21\Big(\frac{\sigma_u^2}{m} +\frac{m\sigma_v^2}{n}\Big)L\gamma + \frac{72L^3\gamma}{mn}\cC^t.
\end{align*}
By averaging it over $t=0,1,\dots,T-1$, we obtain the following inequality:
\begin{equation}\label{eqa-scaffold-step2}
\begin{aligned}
    \Big(\frac{1}{2} \hspace{-0.5mm} - \hspace{-0.5mm} 6LK\gamma\Big)\frac{1}{T}\sum\limits_{t=0}^{T-1}(\cG_u^t \hspace{-0.5mm}+\hspace{-0.5mm} \cG_{\vv}^t) &\le \frac{F_0}{TK\gamma} \hspace{-0.5mm}+\hspace{-0.5mm} 21\Big(\frac{\sigma_u^2}{m}+\frac{m\sigma_v^2}{n}\Big)L\gamma  + \frac{72L^3\gamma}{mnT}\sum\limits_{t=0}^{T-1}\cC^t.
\end{aligned}
\end{equation}
Recalling Lemma \ref{lem-scaffold-c} and employing deductive reasoning, we have
\begin{align*}
\frac{m}{4n}\sum\limits_{t=0}^{T-1}\cC^t &\le \left[ \sum\limits_{t=0}^{T-1} \frac{mn^{2/3}K}{36L^2}(\cG_u^{t} + \hat \cG_{\vv}^{t}) \right] + \frac{m^2T}{18L^2}\Big(\frac{\sigma_u^2}{m} + \frac{m\sigma_v^2}{n}\Big).
\end{align*}
Combining it with \eqref{eqa-scaffold-step2} yields the following inequality:
\begin{align*}
    \Big(\frac{1}{2}-6LK\gamma - \frac{8LK\gamma n^{2/3}}{m}\Big)\frac{1}{T}\sum\limits_{t=0}^{T-1}(\cG_u^t + \cG_{\vv}^t) &\le \frac{F_0}{TK\gamma} + 37\Big(\frac{\sigma_u^2}{m}+\frac{m\sigma_v^2}{n}\Big)L\gamma.
\end{align*}
By utilizing $6LK\gamma \le \frac{1}{8}, 8LK\gamma \le \frac{m}{8n^{2/3}}$, we complete the proof.
\end{proof}

With Theorem \ref{thm-scaffold}, we can obtain the following convergence rate for the Scaffold-P 
algorithm by selecting proper local and global step sizes.
\begin{corollary}[\sc Convergence rate]
\label{cor-scaffold-rate}
 Under the same assumptions as  Theorem \ref{thm-scaffold}, by setting $\gamma = \Big(72LK\max(\frac{n^{2/3}}{m}, 1) + \sqrt{\frac{37LKT}{F_0}(\frac{\sigma_u^2}{m}+\frac{m\sigma_v^2}{n})}\Big)^{-1}$, $\eta_u \ge \sqrt{m}$ and $\eta_v \ge \sqrt{\frac{n}{m}}$, it holds that
\begin{align*}
    \avgt(\cG_u^t+\hat\cG_{\vv}^t)&\le 8\sqrt{\frac{37LF_0}{KT} \Big(\frac{\sigma_u^2}{m}+\frac{\sigma_v^2m}{n}\Big)}+ \frac{288LF_0}{T}\max\Big(1, \frac{n^{2/3}}{m}\Big). 
\end{align*}
\end{corollary}
\begin{remark}[\sc removing gradient dissimilarity influence] 
Corollary \ref{cor-scaffold-rate} presents the convergence rate of Scaffold-P without the need for the bounded gradient dissimilarity assumption. Hence, by removing the influence of $b^2$, Scaffold-P attains a faster and more robust convergence compared to FedAvg-P, particularly in scenarios of high gradient dissimilarity.
\end{remark}

\subsubsection{Convergence rate in special scenarios} 
When the personal variable $v_i$ is removed, Scaffold-P reduces to the original Scaffold algorithm proposed in \cite{karimireddy2020scaffold} for solving the classical FL problem \eqref{prob-full} with a single shared model. In this setting, by taking $\sigma_v^2 = 0$ in Corollary \ref{cor-scaffold-rate}, the rate of the original Scaffold under partial client participation {can be recovered} as follows:
\begin{align}\label{scaffold-ours}
\avgt \EE\|\nabla f(u^t)\|^2  = O\left(\sqrt{\frac{\sigma_u^2}{mKT}} + \frac{1}{T}\max\left(1, \frac{n^{2/3}}{m}\right)\right).
\end{align}
Before our results, the best known rate for Scaffold established in \cite{karimireddy2020scaffold} is:
\begin{align}\label{scaffold-spk}
\avgt \EE\|\nabla f(u^t)\|^2  = O\left(\sqrt{\frac{\sigma_u^2}{mKT}} + \frac{1}{T}\left(\frac{n}{m}\right)^{2/3}\right).
\end{align}
It {can be} observed that our recovered rate improves upon the best known result in the literature, demonstrating the sharpness of our analysis for Scaffold-P.  

When $K=1$ and there is no shared variables $u$,
Scaffold-P reduces to parallel SGD. In this setting, Corollary \ref{cor-scaffold-rate} recovers the rate of parallel SGD by setting $\sigma_u^2 = 0$.

\section{Numerical Experiments}\label{sec-experiment}
In this section, we conduct numerical experiments to validate the main theoretical findings presented in Sections \ref{sec-fedavg} and \ref{sec-scaffold}. In our experiments, we consider the logistic regression problem whose objective function has the form
\begin{align*}
f_i(u, v_i) = \frac{1}{N}\sum_{\ell \in \mathcal{D}_i} \log(1+\exp(-c_\ell(a_\ell^Tu+b_\ell^Tv_i))) + \rho\,r(u, v_i),
\end{align*}
where $r(u, v_i) = \frac{\|u\|^2}{1+\|u\|^2} + \frac{\|v_i\|^2}{1+\|v_i\|^2}$ is a non-convex regularization term \cite{antoniadis2011penalized, xin2021improved}. {Data $\{(a_\ell,b_\ell,c_\ell)\}$ is from the MNIST dataset, which has been divided into $n$ groups, denoted as $\mathcal{D}_1$ through $\mathcal{D}_n$, and distributed to $n$ clients in accordance with the FL framework.} Here, we have set $n=10$. Each individual data point is then split into two components: the shared features $a_\ell\in\mathbb{R}^{d_u}$ and the personal features $b_\ell\in\mathbb{R}^{d_v}$.

We explore the influence of $\gamma, m$, and $K$ on the performance of  FedAvg-P and Scaffold-P, which is depicted in Figs.~\ref{fig-fedavg-gnorm-lom} and \ref{fig-scaffold-gnorm-lom}, respectively. Unless specified otherwise, we set $\gamma = 0.001, m=9, K=25$.
In Figs.~\ref{fig-fedavg-y-lom} and \ref{fig-scaffold-y-lom}, we visualize the influence of the step size parameter $\gamma$ on the convergence behavior. Consistent with our convergence results in Theorems \ref{thm-fedavg} and \ref{thm-scaffold}, these figures illustrate that a smaller value of $\gamma$ leads to a smaller convergence error when the algorithm reaches a steady state, as well as a slower convergence rate. Furthermore, Figs.~\ref{fig-fedavg-m-lom} and \ref{fig-scaffold-m-lom} demonstrate the impact of the number of participating clients $m$ on the gradient norm. 
In this experiment, we intentionally set $\sigma_v$ to be 0, indicating that there is no noise present in the computation of the gradient with respect to $v_i$. This deliberate choice allows us to accurately examine the impact of $m$ on the convergence rate. The curves reveal that, with a fixed step size, increasing the number of participating clients results in smaller convergence error, which aligns with the findings from our convergence results.

\begin{figure*}[!htbp]
\centering
\subfloat[different $\gamma$.]{\includegraphics[trim={0 0 0 1.4cm},clip,width=1.9in]{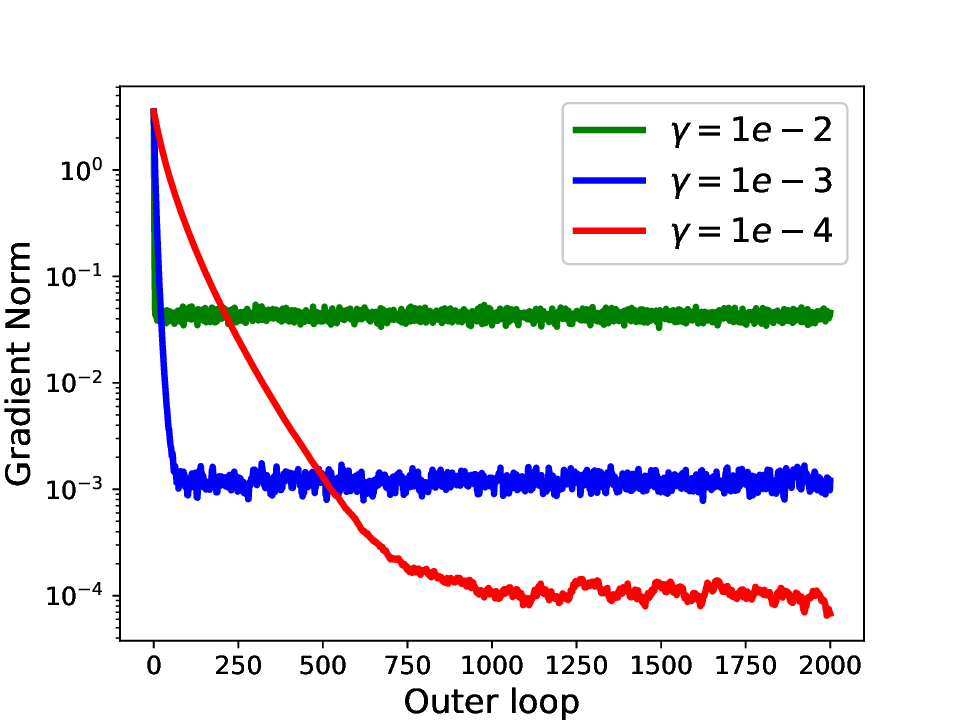}%
\label{fig-fedavg-y-lom}}%
\hfil
\subfloat[different $m$.]{\includegraphics[trim={0 0 0 1.4cm},clip,width=1.9in]{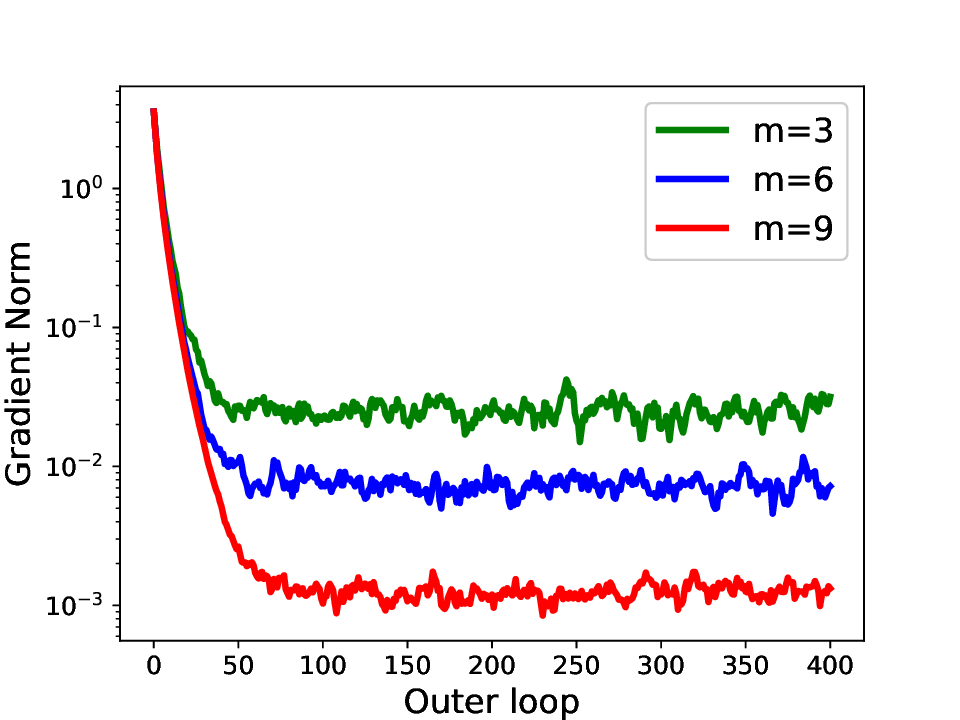}%
\label{fig-fedavg-m-lom}}%
\hfil
\subfloat[different $K$.]{\includegraphics[trim={0 0 0 1.4cm},clip,width=1.9in]{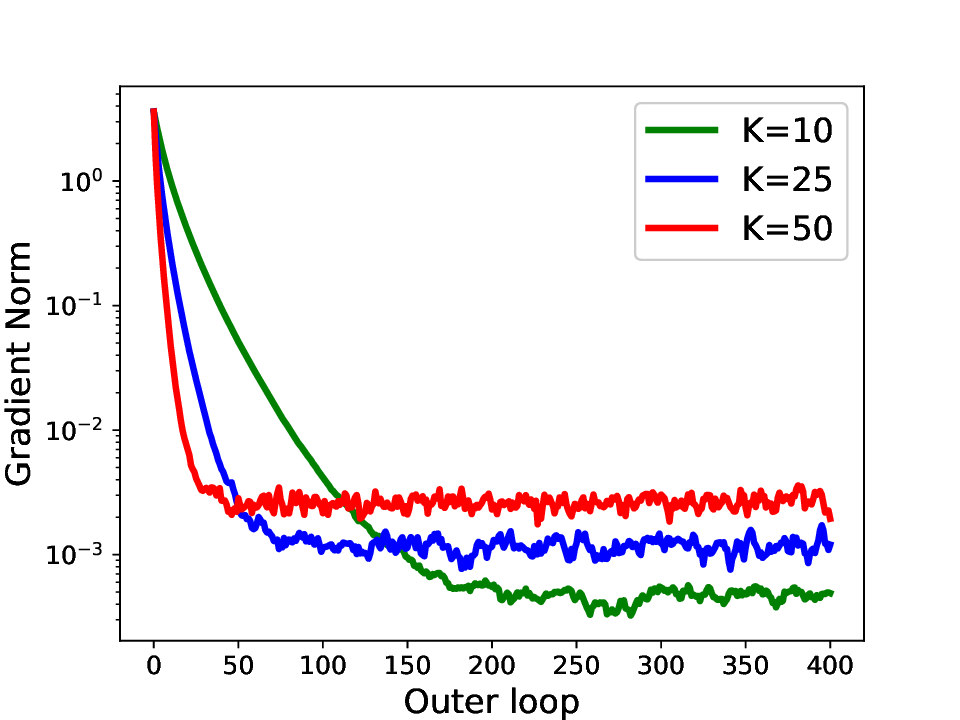}%
\label{fig-fedavg-K-lom}}%
\hfil
\subfloat[different $K$, m=n.]{\includegraphics[trim={0 0 0 1.4cm},clip,width=1.9in]{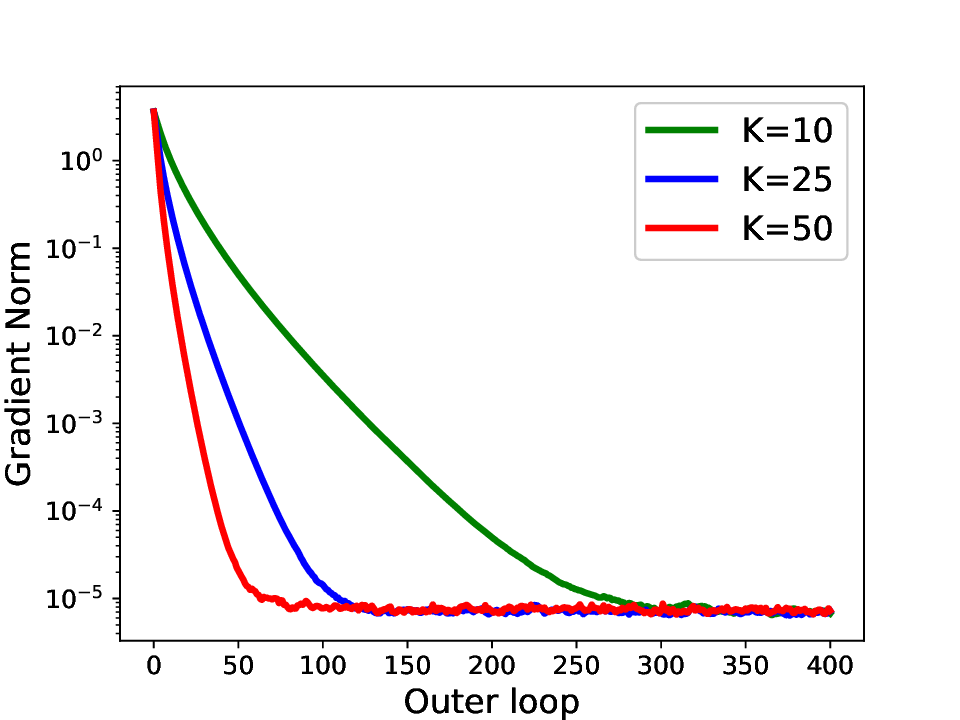}%
\label{fig-fedavg-K-lomd}}%
\caption{Gradient norm versus the number of outer iteration for FedAvg-P algorithm.}
\label{fig-fedavg-gnorm-lom}
\end{figure*}

\vspace{-1ex}
\begin{figure*}[!htbp]
\centering
\subfloat[different $\gamma$.]{\includegraphics[trim={0 0 0 1.4cm},clip,width=1.7in]{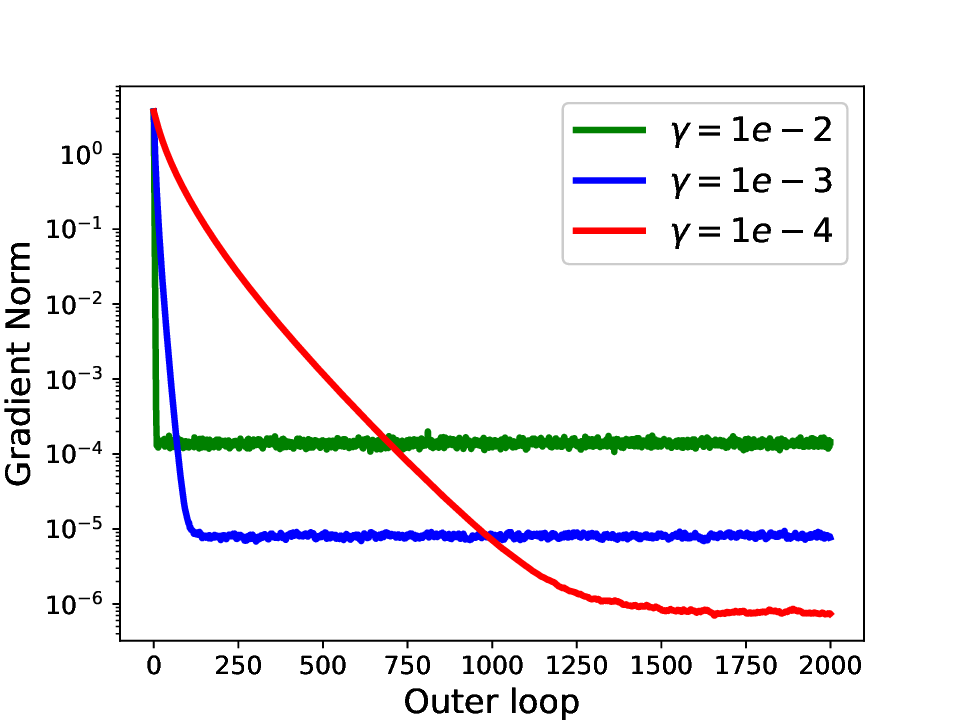}%
\label{fig-scaffold-y-lom}}
\hfil
\subfloat[different $m$.]{\includegraphics[trim={0 0 0 1.4cm},clip,width=1.7in]{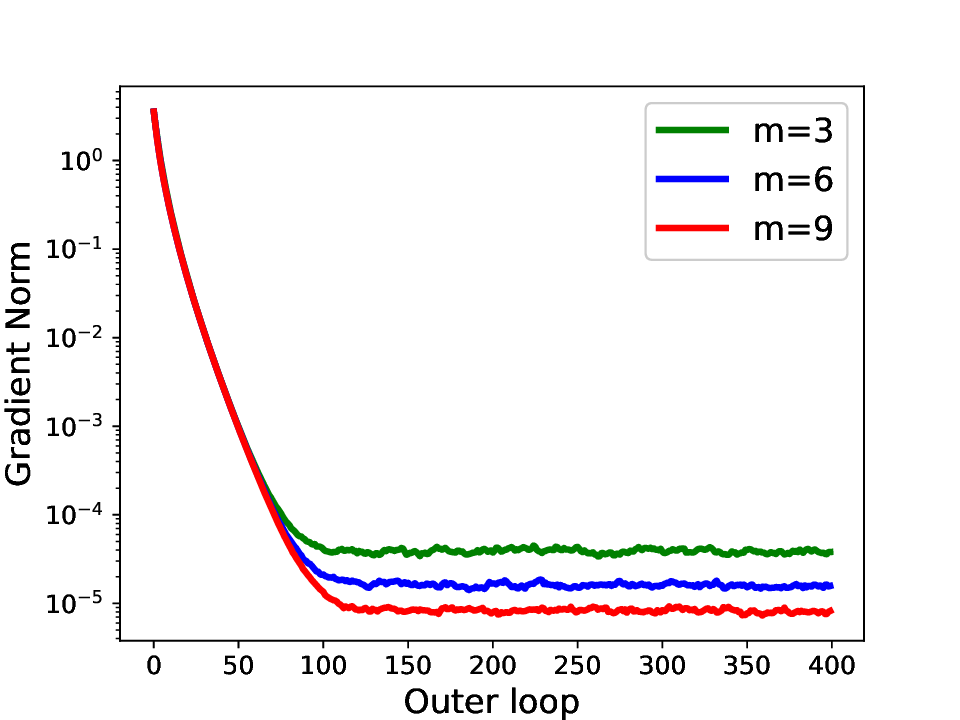}%
\label{fig-scaffold-m-lom}}
\hfil
\subfloat[different $K$.]{\includegraphics[trim={0 0 0 1.4cm},clip,width=1.7in]{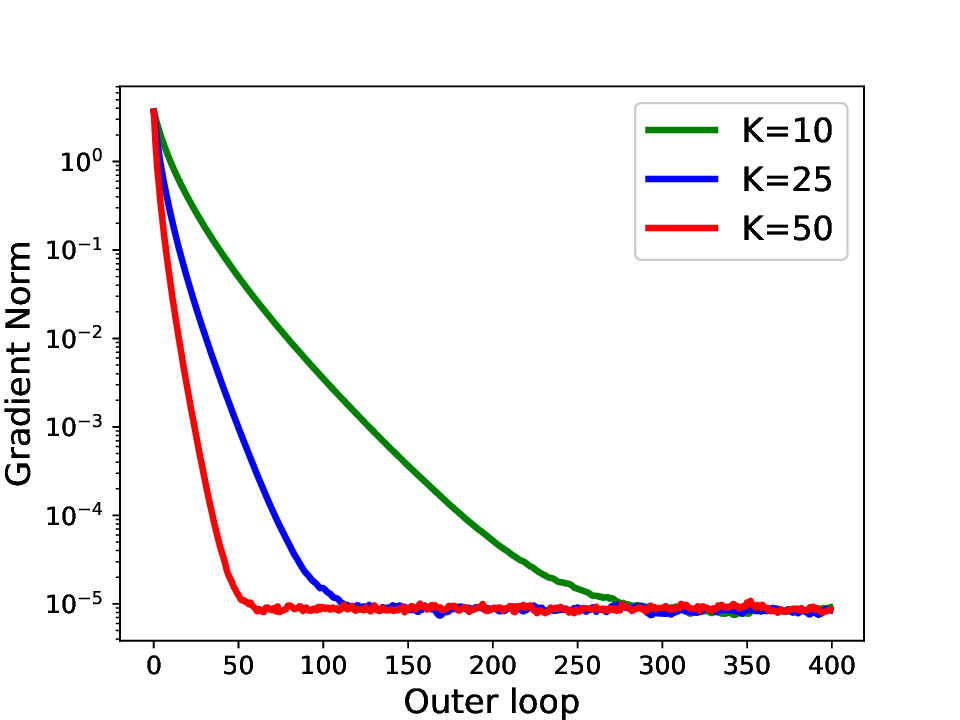}%
\label{fig-scaffold-K-lom}}
\hfil
\caption{Gradient norm versus the number of outer iteration for Scaffold-P algorithm.}
\label{fig-scaffold-gnorm-lom}
\end{figure*}

Figs.~\ref{fig-fedavg-K-lom} and \ref{fig-scaffold-K-lom} provide valuable insights into the influence of the number of local updates $K$ on the convergence behavior of the FedAvg-P and Scaffold-P algorithms. The experimental results align with the theoretical findings presented in Theorems \ref{thm-fedavg} and \ref{thm-scaffold}. Specifically, increasing the value of $K$ accelerates the convergence of both algorithms, confirming the theoretical analysis. However, it is important to note the contrasting impact of $K$ on the convergence error behavior of the FedAvg-P and Scaffold-P algorithms. According to Theorem \ref{thm-fedavg}, in the scenario of partial client participation, a larger value of $K$ leads to a larger gradient norm in the steady state for the FedAvg-P algorithm {due to the term $O(\frac{(n-m)Kb^2\gamma}{mn} + \frac{K^2b^2\gamma^2}{\eta_u^2})$.} This implies a trade-off between convergence rate and the convergence error when adjusting parameter $K$. Conversely, there are no such terms impacting the convergence rate of Scaffold-P. Thus, the values of $K$ and $b^2$ do not influence the steady-state convergence error. These observations are distinctly illustrated in the provided figures. Furthermore, Fig.~\ref{fig-scaffold-K-lom} shows that Scaffold-P converges with a much smaller error than FedAvg-P in the steady state, which illustrates how removing $b^2$ can benefit the convergence performance as predicted by Theorem \ref{thm-scaffold}. 

Fig.~\ref{fig-fedavg-K-lomd} examines how the number of local updates, denoted as $K$, affects the convergence of FedAvg-P in the scenario of full client participation when $m=n$. According to Theorem \ref{thm-fedavg-f}, the steady-state convergence error is independent of $K$ and $b^2$, which is well justified in Fig.~\ref{fig-fedavg-K-lomd}.

\section{Conclusion}

This paper introduces FedAvg-P and Scaffold-P and offers a rigorous convergence analysis. Our analysis framework is versatile, making it applicable to a wide range of FL problems. When we apply our analysis to the classical FL problem with a single shared model, it yields more refined convergence rates compared to existing approaches. Moreover, in the case of full client participation, we demonstrate that the BGD assumption for the FedAvg algorithm can be omitted. Future directions include introducing compressed or decentralized communication to the partial model personalization problem. 

\appendix
\section{Proofs}\label{sec-proof}
In this section, we provide detailed proofs for several lemmas used in this paper. We begin by introducing some prerequisite technical lemmas. The following one is frequently employed in the analysis of stochastic first-order algorithms, whose proof derives from the property of $L^2$ martingales \cite{durrett2019probability}.
\begin{lemma}\label{lem-mds}
Suppose that $\{X_n, \cF_n\}$ is a $L^2$ martingale difference sequence, then
	\begin{align*}
		&\EE\|X_1+X_2+...+X_n\|^2  = \EE \|X_1\|^2+\EE \|X_2\|^2+...+\EE \|X_n\|^2.
    \end{align*}
\end{lemma}

The following lemma is established to deal with the partial client participation.
\begin{lemma}\label{lem-id}
Let $\{X_i:i=1,2,...,n\}$ be random vectors that are independent of $\{\cht:i=1,2,...,n\}$, then we have  
    \begin{align*}
        \EE\Big\|\frac{1}{m}\sumi\cht X_i\Big\|^2 = \EE\|\bar X\|^2 + \frac{2(n-m)}{mn^2}\sumi\EE\|X_i-\bar X\|^2,\  \text{where}\ \bar X := \avgi X_i.
    \end{align*}
   
\end{lemma}
\begin{proof}
If $n=1$, this lemma holds evidently. Now we assume $n\ge 2$ and obtain
\begin{align*}
\EE\Big\|\frac{1}{m}\sumi\cht X_i\Big\|^2 
&= \frac{1}{m^2}\sum\limits_{i,j}\EE(\cht\ch_{\{j\in \cS^t\}})\EE\langle X_i, X_j \rangle.\\
\end{align*}
The equation derives from the independence between $X_i$ and $\cht$. Recalling that $\cS^t$ is a randomly selected subset of $\{1,2,...,n\}$ with size $m$, we have
\begin{align*}
\EE\Big\|\frac{1}{m}\sumi\cht X_i\Big\|^2 
&= \frac{1}{mn}\sumi\EE\|X_i\|^2 + \frac{m-1}{mn(n-1)}\sum\limits_{i\ne j}\langle X_i, X_j \rangle\\
&= \EE\|\bar X\|^2 + \frac{n-m}{mn(n-1)}\sumi\EE\|X_i-\bar X\|^2. 
\end{align*}
Finally, use the fact that $\frac{1}{n-1} \le \frac{2}{n}$ when $n \ge 2$ to obtain the final result. 
\end{proof}

The next lemma provides an upper bound on the difference between the function values at two consecutive outer loops.
\begin{lemma}\label{lem-step}
	Under Assumption \ref{ass-smooth}, it holds that 
\begin{equation}\label{eqa-step}
    \begin{aligned}
    \EE f(u^{t+1},\vv^{t+1}) - \EE f(u^t,\vv^t) &\le \EE \langle \nabla_u f(u^t,\vv^t), \Delta u^t \rangle + \avgi \EE \langle \nabla_v f_i(u^t, v_i^t), \Delta v_i^t \rangle \\
    &\quad \quad + L\EE\|\Delta u^t\|^2 + \avgi L\EE\|\Delta v_i^t\|^2,
    \end{aligned}
\end{equation}
\end{lemma}
where $\Delta u^t := u^{t+1} - u^t$ and $\Delta v_i^t := v_i^{t+1} - v_i^t$.
\begin{proof}
Firstly, we can establish the following inequality since each $f_i$ has a Lipschitz continuous gradients with respect to both $u$ and $v_i$ by Assumption \ref{ass-smooth}: 
\begin{align*}
 f_i(u^{t+1}, &v_i^{t+1}) - f_i(u^t, v_i^t) = f_i(u^{t+1}, v_i^{t+1}) - f_i(u^t, v_i^{t+1}) + f_i(u^t, v_i^{t+1}) - f_i(u^t, v_i^t)\\
&\le \langle \nabla_u f_i(u^t, v_i^{t+1}), \Delta u^t \rangle+ \frac{L_u}{2}\|\Delta u^t\|^2 + \langle \nabla_v f_i(u^t, v_i^t), \Delta v_i^t \rangle + \frac{L_v}{2}\|\Delta v_i^t\|^2.
\end{align*}
By leveraging Assumption \ref{ass-smooth}, we have 
\begin{align*}
\langle \nabla_u f_i(u^t, v_i^{t+1}), \Delta u^t \rangle 
&= \langle \nabla_u f_i(u^t, v_i^{t+1}) - \nabla_u f_i(u^t, v_i^t), \Delta u^t \rangle  + \langle \nabla_u f_i(u^t, v_i^t),  \Delta u^t \rangle\\
&\le L_{uv}\|\Delta v_i^t\|\|\Delta u^t\|+ \langle \nabla_u f_i(u^t, v_i^t),  \Delta u^t \rangle\\
&\le \frac{L_{uv}}{2}(\|\Delta v_i^t\|^2+\|\Delta u^t\|^2)  + \langle \nabla_u f_i(u^t, v_i^t),  \Delta u^t \rangle.
\end{align*}
Combine the above two inequalities to obtain 
\begin{align*}
 f_i(u^{t+1}, v_i^{t+1}) - f_i(u^t, v_i^t) &\le \langle \nabla_u f_i(u^t, v_i^t), \Delta u^t \rangle + \langle \nabla_v f_i(u^t, v_i^t), \Delta v_i^t \rangle\\
&\quad + L\|\Delta u^t\|^2 + L\|\Delta v_i^t\|^2.
\end{align*}
Then the lemma can be deduced by considering relationship  $f(u,\vv) = \frac{1}{n} \sum_i f_i(u,v_i)$ and taking the expected value.
\end{proof}

Having completed the necessary preparatory work, we now present the detailed proofs of several lemmas in Sections \ref{sec-fedavg} and \ref{sec-scaffold}.
\begin{proof}[\underline{Proof of Lemma \ref{lem-fedavg-step}}]
The proof of this lemma relies on the application of Lemma \ref{lem-step} in this specific case. We will bound the four terms appearing on the right-hand side of \eqref{eqa-step} individually to establish the desired result. We start by introducing the following two $\sigma$-algebra to facilitate our analysis:
\begin{align*}
&\cF^t:= \sigma\left(\{\xi_{i,k}^r, \cS^r:1\le i\le n, 0 \le k\le K-1, 0\le r\le t-1\}\right),\\
&\cF^t_{i,k} = \sigma\left(\cF^t\cup\{\cS^t\}\cup\{\xi_{i,l}^t:l \le k\}\cup\{\xi_{j,l}^t:j\le i-1,l \le K-1\}\right).
\end{align*}
From the definition of $\Delta u$ and $\Delta v_i$, we obtain
\begin{align*}
\Delta u^t &= u^{t+1}-u^t = \frac{\eta_u}{m}\sumpit(\hat u_{i,K}^t-\hat u_{i,0}^t)=-\frac{\gamma}{m}\sumik\cht\du F(\hat u_{i,k}^t, \hat v_{i,k}^t, \xi_{i,k}^t),\\
\Delta v_i^t &= v_i^{t+1}-v_i^t = \eta_v(\hat v_{i,K}^{t}-\hat v_{i,0}^t)\cht =\! -\gamma\cht\sumk\dv F(\hat u_{i,k}^t,\! \hat v_{i,k}^t,\! \xi_{i,k}^t), 
\end{align*}
where $\cht$ is the characteristic function of $\cS^t$. Thus it holds that
\begin{equation}\label{eqa-step1}
\begin{aligned}
 \EE\langle \du f(u^t, \vv^t), \Delta u^t \rangle &= \EE\bigg\oldlangle \du f(u^t, \vv^t), -\frac{\gamma}{m}\sumik\cht\du F(\hat u_{i,k}^t, \hat v_{i,k}^t, \xi_{i,k}^t)\bigg\oldrangle\\
&= \sumik\EE\ \EE(\oldlangle \du f(u^t,\vv^t),\!  -\frac{\gamma}{m}\cht\du F(\hat u_{i,k}^t,\! \hat v_{i,k}^t,\! \xi_{i,k}^t)\oldrangle|\cF_{i,k-1}^t)\\
&= \sumik\EE\langle\du f(u^t, \vv^t),\!-\frac{\gamma}{m}\cht\EE(\du F(\hat u_{i,k}^t,\! \hat v_{i,k}^t,\! \xi_{i,k}^t)|\cF_{i,k-1}^t)\rangle\\
&= -\frac{\gamma}{m}\EE\cht\sumik\EE\langle \du f(u^t, \vv^t), \du f_i(\hat u_{i,k}^t, \hat v_{i,k}^t)\rangle = T_{1,u}\\
\end{aligned}   
\end{equation}
The second equality above holds by virtue of the ``tower property" of conditional expectation, and the third equality holds because $\cht$ and $\du f(u^t, \vv^t)$ are $\cF_{i,k-1}^t$-measurable. By employing similar deductions, we can derive the following equation: 
\begin{equation*}
\EE\langle \dv f_i(u^t, v_i^t), \Delta v_i^t \rangle = -\frac{\gamma m}{n}\sumk\EE\langle \dv f_i(u^t, v_i^t), \dv f_i(\hat u_{i,k}^t, \hat v_{i,k}^t)\rangle, 
\end{equation*}
which implies $\frac{1}{n}\sum_i \langle \nabla_v f_i(u^t,v_i^t), 
\Delta v_i^t \rangle = T_{1,\vv}$. 
To establish bounds for $\EE\|\Delta u^t\|^2$ and $\EE\|\Delta v_i^t\|^2$, we introduce the following sequences:
\begin{align*}
&Y_{i,k}^t :=  \cht\left(\du F(\hat u_{i,k}^t, \hat v_{i,k}^t, \xi_{i,k}^t)-\du f_i(\hat u_{i,k}^t, \hat v_{i,k}^t)\right), \\
&Z_{i,k}^t :=  \cht\left(\dv F(\hat u_{i,k}^t, \hat v_{i,k}^t, \xi_{i,k}^t)-\dv f_i(\hat u_{i,k}^t, \hat v_{i,k}^t)\right).
\end{align*}
Since $Y_{i,k}^t$ is adapted to $\cF_{i,k}^t$, and $\cS^t$ is measurable in $\cF^t_{i,k}$ for all $i,k$, 
\begin{align*}
\EE(Y_{i,k}^t|\cF_{i,k-1}^t)
    &=\cht [\EE(\du F(\hat u_{i,k}^t, \hat v_{i,k}^t,\xi_{i,k}^t)|\cF^t_{i,k-1})-\du f_i(\hat u_{i,k}^t, \hat v_{i,k}^t)] = 0,\\
\EE(Y_{i,0}|\cF_{i-1, K-1}^t)
    &=\cht[\EE(\du F(u^t, v_i^t,\xi_{i,0}^t)|\cF^t)-\du f_i(u^t, v_i^t)] = 0.
\end{align*}
Therefore, the sequence \( (Y_{i,k}^t, \mathcal{F}_{i,k}^t)_{1\leq i\leq n, 0\leq k\leq K-1} \) can be regarded as a martingale difference sequence, with the ordering defined such that \( (i,k) < (j,l) \) if \( i = j \) or both \( i = j \) and \( k < l \). The same holds true for $(Z_{i,k}^t, \cF_{i,k}^t)_{0\le k\le K-1}$. 
Thus, 
\begin{equation}\label{ineq-du1}
\begin{aligned}
\EE\|\Delta u^t\|^2
&= \EE\Big\|\frac{\gamma}{m}\sumik\cht\du F(\hat u_{i,k}^t, \hat v_{i,k}^t, \xi_{i,k}^t)\Big\|^2\\
&\le 2\EE\Big\|\frac{\gamma}{m}\sumik\cht\du f_i(\hat u_{i,k}^t, \hat v_{i,k}^t)\Big\|^2 +\frac{2\gamma^2}{m^2} \sumik\EE\|Y_{i,k}^t\|^2 \\
&\le 2T_{2,u} + \frac{4\gamma^2(n-m)}{mn^2} T_{3,u} + \frac{2K\gamma^2\sigma_u^2}{m}, 
\end{aligned}
\end{equation}
where the first inequality holds due to the mean inequality and Lemma~\ref{lem-mds}. In the last step, we use Lemma~\ref{lem-id} to bound the first term: 
\[ \EE\Big\|\frac{\gamma}{m}\sumik\cht\du f_i(\hat u_{i,k}^t, \hat v_{i,k}^t)\Big\|^2 \le T_{2,u} + \frac{2\gamma^2(n-m)}{mn^2} T_{3,u}; 
\]
and use the independence of the virtual sequence and the characteristic function, along with Assumption \ref{ass-sto} to bound the second term: 
\begin{align*}
\EE\|Y_{i,k}^t\|^2
&= \EE\cht\EE\|\du  F(\hat u_{i,k}^t, \hat v_{i,k}^t,\xi_{i,k}^t) -\du f_i(\hat u_{i,k}^t, \hat v_{i,k}^t)\|^2
\le \frac{\sigma_u^2m}{n}. 
\end{align*}

With a similar analysis approach, we can also establish the inequalities 
\begin{equation}\label{ineq-dv}
\begin{aligned}
\EE\|\Delta v_i^t\|^2 
&\le \frac{2\gamma^2m}{n}\EE\bigg\|\sumk \dv f_i(\hat u_{i,k}^t, \hat v_{i,k}^t)\bigg\|^2 + \frac{2K\gamma^2\sigma_v^2m}{n}, 
\end{aligned}
\end{equation}
for all $i\in[n]$, which implies $\frac{1}{n}\sum_i \|\Delta v_i^t\|^2 \le 2T_{2,\vv} + \frac{2K\sigma_v^2\gamma^2m}{n}$. 
Apply the inequalities derived above to \eqref{eqa-step} to obtain \eqref{lem-fedavg-step}, and the proof is complete.
\end{proof}
\begin{proof}[\underline{Proof of Lemma \ref{lem-fedavg-b}}]
The inequality $\cB^t \le n\cG_u^t + nb^2$ follows directly from Assumption \ref{ass-hete}. Now, let Assumptions \ref{ass-smooth}, \ref{ass-sto}, and $m=n$ hold, we first prove the following inequality:
\begin{equation}\label{ineq-be}
\begin{aligned}
\cB^t &\le \Big(1+\frac{1}{T-1}\Big)\cB^{t-1} + nT(\cG_u^{t-1}+\cG_{\vv}^{t-1}) + \frac{TL^2}{K}\cE^{t-1} + \frac{nT}{K}\Big(\frac{\sigma_u^2}{n}+\sigma_v^2\Big).
\end{aligned}
\end{equation}
Recalling $\cB^t = \sum_{i=1}^n\EE \|\nabla_u f_i(u^t,v_i^t)\|^2$, we begin by establishing upper bounds on the individual terms $\EE\|\du f_i(u^t, v_i^t)\|^2$. By the convexity of $\|\cdot\|^2$, 
\begin{align*}
\EE\|\du f_i(u^t, v_i^t)\|^2 &= \EE\|\du f_i(u^t, v_i^t) - \du f_i(u^{t-1}, v_i^{t-1}) + \du f_i(u^{t-1}, v_i^{t-1})\|^2\\
&\le \Big(1+\frac{1}{\kappa}\Big)\EE\|\du f_i(u^{t-1}, v_i^{t-1})\|^2 \\
&\quad + 2(1+\kappa) L^2\EE(\|\Delta u^{t-1}\|^2+\|\Delta v_i^{t-1}\|^2)
\end{align*}
for all $i\in[n]$ and all $\kappa>0$. 
Now, by setting $\kappa = T-1$, we have
\begin{equation}\label{ineq-bd}
\begin{aligned}
\cB^t \le \Big(1+\frac{1}{T-1}\Big)\cB^{t-1} + 2TL^2\sumi\EE\left(\|\Delta u^{t-1}\|^2+\|\Delta v_i^{t-1}\|^2\right).
\end{aligned}
\end{equation}
Apply Lemma~\ref{lem-T2uv} to \eqref{ineq-du1} and \eqref{ineq-dv} to obtain 
\begin{align*}
\EE\|\Delta u^t\|^2 &\le 4K^2\gamma^2\cG_u^t + \frac{8L^2K\gamma^2}{n}\cE^t  + \frac{2K\gamma^2\sigma_u^2}{n},\\
\avgi \EE \|\Delta v_i^t\|^2 &\le 4K^2\gamma^2\cG_{\vv}^t + \frac{8L^2K\gamma^2}{n}\cE^t + 2K\sigma_v^2\gamma^2,
\end{align*}
where we use the fact that $m=n$. These two inequalities can turn \eqref{ineq-bd} into
\begin{align*}
\cB^t &\le \Big(1+\frac{1}{T-1}\Big)\cB^{t-1} + 8nTL^2K^2\gamma^2(\cG_u^{t-1}+ \cG_{\vv}^{t-1}) \\
&\quad +32TL^4K\gamma^2\cE^{t-1}+ 4nTL^2K\gamma^2\Big(\frac{\sigma_u^2}{n} + \sigma_v^2\Big).
\end{align*}
Then, the inequality \eqref{ineq-be} can be proved using $8LK\gamma \le 1$. Next, we use Lemma \eqref{lem-fedavg-e} to replace $\cE^{t-1}$. From Lemma \eqref{lem-fedavg-e}, we have
\begin{align*}
    \cE^{t-1} &\le 8K^3\gamma_u^2\cB^{t-1} + 8nK^3\gamma_v^2\cG_{\vv}^{t-1} + 4nK^2(\sigma_u^2\gamma_u^2+\sigma_v^2\gamma_v^2)\\
    &\le  \frac{8K^3\gamma^2}{T^2}\cB^{t-1} + 8nK^3\gamma^2\cG_{\vv}^{t-1}+4nK^2\gamma^2\Big(\frac{\sigma_u^2}{n}+\sigma_v^2\Big)\\
    &\le  \frac{K}{L^2T^2}\cB^{t-1} + \frac{nK}{L^2}\cG_{\vv}^{t-1}+\frac{n}{L^2}\Big(\frac{\sigma_u^2}{n}+\sigma_v^2\Big).
\end{align*}
The second inequality derives from $\eta_u\ge \max(T, \sqrt{n}), \eta_v \ge 1$, and the final inequality holds since $8LK\gamma \le 1$. Together with \eqref{ineq-be}, we complete the proof.
\end{proof}

\begin{proof}[\underline{Proof of Lemma \ref{lem-scaffold-step}}]
Recalling the definition of $\alpha_{i,k}^{t}, \beta_{i,k}^{t}$, we have
\begin{align*}
c_i^t = \avgk\du F(\alpha_{i,k}^{t-1}, \beta_{i,k}^{t-1}, \xi_{i,k}^{t-1}),\ c^t = \avgi c_i^t = \avgik\du F(\alpha_{i,k}^{t-1}, \beta_{i,k}^{t-1}, \xi_{i,k}^{t-1}).
\end{align*}

Similarly with the proof of Lemma \ref{lem-fedavg-step}, this one utilize Lemma \ref{lem-step} and bound the four terms in the right hand of \eqref{eqa-step}. Note that
\begin{align*}
\Delta u^t &= u^{t+1}-u^t = -\frac{\gamma}{m}\sumik\cht(\du F(\hat u_{i,k}^t, \hat v_{i,k}^t, \xi_{i,k}^t) - c_i^t + c^t),\\
\Delta v_i^t &= v_i^{t+1}- v_i^t = -\gamma\cht\sumk\dv F(\hat u_{i,k}^t, \hat v_{i,k}^t, \xi_{i,k}^t).
\end{align*}
By employing an analytical methodology analogous to that used in equation \eqref{eqa-step1} and observing that 
$\sum_i c_i^t = \sum_i c^t$, we can establish
\[ 
    \EE\langle \du f(u^t, \vv^t), \Delta u^t \rangle = T_{1,u}
    \quad \text{and} \quad
    \avgi \EE\langle \dv f_i(u^t, v_i^t), \Delta v_i^t \rangle = T_{1,\vv}. 
\]

Next, we bound the terms $\EE\|\Delta u^t\|^2$ and $\EE\|\Delta v_i^t\|^2$. For $\EE\|\Delta u^t\|^2$, we first introduce the sequences 
\begin{align*}
 W_{i,k}^t &:= \cht\Big(\du F(\alpha_{i,k}^{t-1}, \beta_{i,k}^{t-1}, \xi_{i,k}^{t-1})-\du f_i(\alpha_{i,k}^{t-1}, \beta_{i,k}^{t-1})\Big),\\
\hat W_{i,k}^t &:= \du F(\alpha_{i,k}^{t-1}, \beta_{i,k}^{t-1}, \xi_{i,k}^{t-1})-\du f_i(\alpha_{i,k}^{t-1}, \beta_{i,k}^{t-1}).
\end{align*}
A proof structure similar to that used for $Y_{i,k}^t$ can be applied to demonstrate that $W_{i,k}^t, \hat W_{i,k}^t$ also constitute martingale difference sequences, and satisfy the condition $\EE\|W_{i,k}^t\|^2 \le \frac{\sigma_u^2m}{n}$ and $\EE\|\hat W_{i,k}^t\|^2 \le \sigma_u^2$.
Then by combining with Assumption \ref{ass-sto} and the definition of $d_i^t, d^t$, we obtain 
\begin{equation}\label{ineq-cd}
\begin{aligned}
    \EE\Big\|\sumik\cht(c_i^t-d_i^t)\Big\| &= \EE\Big\|\sumik W_{i,k}^t\Big\|^2 = \sumik\EE\|W_{i,k}^t\|^2 \le mK\sigma_u^2,\\
    \EE\Big\|\sumik\cht(c^t-d^t)\Big\| &= \EE\cht \EE\Big\|\sumik \hat W_{i,k}^t\Big\|^2 = \frac{m}{n}\sumik\EE\|\hat W_{i,k}^t\|^2 \le mK\sigma_u^2.
\end{aligned}
\end{equation}
Then we can proceed as follows:
\begin{equation}\label{ineq-dus}
\begin{aligned}
\EE\|\Delta u^t\|^2 
&= \EE\Big\|\frac{\gamma}{m}\sumik\cht(\du F(\hat u_{i,k}^t, \hat v_{i,k}^t, \xi_{i,k}^t) - c_i^t + c^t)\Big\|^2\\
&\le 2\EE\Big\|\frac{\gamma}{m}\sumik\cht(\du f_i(\hat u_{i,k}^t, \hat v_{i,k}^t) - d_i^t + d^t)\Big\|^2 \\
&\quad + \frac{2\gamma^2}{m^2}\EE\Big\|\sumik Y_{i,k}^t + \sumik\cht(c_i^t-d_i^t) + \sumik\cht(c^t-d^t) \Big\|^2\\
&\le 2\EE\Big\|\frac{\gamma}{m}\sumik\cht\left(\du f_i(\hat u_{i,k}^t, \hat v_{i,k}^t) - d_i^t + d^t\right)\Big\|^2 +  \frac{18K\gamma^2\sigma_u^2}{m}\\
&\le 2T_{2,u} + \frac{4\gamma^2(n-m)}{mn^2} S_{3,u} +  \frac{18K\gamma^2\sigma_u^2}{m}, 
\end{aligned}
\end{equation}
where the second inequality holds due to the mean inequality, the property of sequence $Y_{i,k}^t$ and \eqref{ineq-cd}. In the final step, we apply Lemma \ref{lem-id} to bound the first term:
\[ \EE\Big\|\frac{\gamma}{m}\sumik\cht\left(\du f_i(\hat u_{i,k}^t, \hat v_{i,k}^t) - d_i^t + d^t\right)\Big\|^2 \le T_{2,u} + \frac{2\gamma^2(n-m)}{mn^2} S_{3,u}. 
\] 
As for $\EE\|\Delta v_i^t\|^2$, since the expression for $\Delta v_i^t$ remains the same as in the FedAvg-P case, the upper bound given in \eqref{ineq-dv} still holds. Thus combining the inequalities proved above with Lemma \ref{lem-step}, we complete the proof.
\end{proof}

\begin{proof}[\underline{Proof of Lemma \ref{lem-scaffold-e}}]
For $k=0,1,...,K-1$, it holds that:
    \begin{align*}
        \EE\|\hat u_{i,k+1}^t-u^t\|^2
        &= \EE\|\hat u_{i,k}^t-\gamma_u(\du F(\hat u_{i,k}^t, \hat v_{i,k}^t, \xi_{i,k}^t) - c_i^t + c^t)- u^t\|^2\\
        &= \EE\|u_{i,k}^t-\gamma_u(\du f_i(\hat u_{i,k}^t, \hat v_{i,k}^t) - c_i^t + c^t)- u^t\|^2 + \sigma_u^2\gamma_u^2\\
        &\le (1\!+\!\frac{1}{K-1})\EE\|\hat u_{i,k}^t\!-\!u^t\|^2\! +\!K\gamma_u^2\EE\|\du f_i(\hat u_{i,k}^t,\! \hat v_{i,k}^t)\! -\! c_i^t\! +\! c^t\|^2\!\!+\!\sigma_u^2\gamma_u^2\\
        &\le (1\!+\!\frac{1}{K-1})\EE\|\hat u_{i,k}^t\!-\!u^t\|^2\! +\! 6K\gamma_u^2\EE\|\du f_i(\hat u_{i,k}^t, \hat v_{i,k}^t)\! - \du f_i(u^t, v_i^t)\|^2 \\
        &\quad + 6K\gamma_u^2(\EE\|\du f_i(u^t, v_i^t) - d_i^t\|^2 + \EE\|\du f(u^t, \vv^t) - d^t\|^2)\\
        &\quad + 6K\gamma_u^2(\EE\|\du f(u^t, \vv^t)\|^2 + \EE\|c_i^t-d_i^t\|^2 + \EE\|c^t - d^t\|^2) + \sigma_u^2\gamma_u^2 \\
    \end{align*}
    The final inequality derives from the mean inequality. Considering the relationships given by
    \begin{align*}
        \EE\|c_i^t-d_i^t\|^2 &= \EE\Big\|\avgk\hat W_{i,k}^t\Big\|^2 = \frac{1}{K^2}\sumk\EE\|\hat W_{i,k}^t\|^2 \le \frac{\sigma_u^2}{K},\\
        \EE\|c^t-d^t\|^2 &= \EE\Big\|\avgik\hat W_{i,k}^t\Big\|^2 = \frac{1}{n^2K^2}\sumik\EE\|\hat W_{i,k}^t\Big\|^2 \le \frac{\sigma_u^2}{nK},
    \end{align*}
    where we utilize the fact that $\hat W_{i,k}^t$ constitutes a martingale difference sequence. In conjunction with Assumption \ref{ass-smooth}, we have
    \begin{align*}
        \EE\|\hat u_{i,k+1}^t-u^t\|^2
        &\le (1+\frac{1}{K-1}+12L^2K\gamma_u^2)\EE\|\hat u_{i,k}^t-u^t\|^2 + 12L^2K\gamma_u^2\EE\|\hat v_{i,k}^t-v_i^t\|^2\\ 
        &\quad+12L^2\gamma_u^2\sumk\cC_{i,k}^t+\frac{12L^2\gamma_u^2}{n}\cC^t + 6K\gamma_u^2\cG_u^t + 13\sigma_u^2\gamma_u^2,
        \end{align*}
        where $\cC_{i,k}^t := \EE(\|\alpha_{i,k}^{t-1}-u^t\|^2 +\|\beta_{i,k}^{t-1}-v_i^t\|^2)$. Since the update rule for $v_i$ is the same as the FedAvg-P algorithm, we can derive the following expression based on the proof of Lemma \ref{lem-fedavg-e}.
        \begin{align*}
        \EE\|\hat v_{i,k+1}^t-v_i^t\|^2
        &\le (1+\frac{1}{K-1}+4L^2K\gamma_v^2)\EE\|\hat v_{i,k}^t-v_i^t\|^2\\
        &\quad + 4L^2K \gamma_v^2\EE\|\hat u_{i,k}^t-u_i^t\|^2+2K\gamma_v^2\EE\|\dv f_i(u^t, v_i^t)\|^2+\sigma_v^2\gamma_v^2.
    \end{align*}
    Recalling that $\cE_{i,k}^t = \EE(\|\hat u_{i,k}^t-u^t\|^2+\|\hat v_{i,k}^t-v_i^t\|^2)$, we have
\begin{equation}\label{eqa-scaffold-ce}
\begin{aligned}
    \cE_{i,k+1}^t &\le (1+\frac{1}{K-1} + 12L^2K(\gamma_u^2+\gamma_v^2))\cE_{i,k}^t + 12L^2\gamma_u^2\sumk\cC_{i,k}^t +\frac{12L^2\gamma_u^2}{n}\cC^t \\
    &\quad+ 6K\gamma_u^2\cG_u^t +2K\gamma_v^2\|\dv f_i(u^t, v_i^t)\|^2+13\sigma_u^2\gamma_u^2 +\sigma_v^2\gamma_v^2.
\end{aligned}
\end{equation}
 We now choose $\gamma_u, \gamma_v$ such that $12L^2K(\gamma_u^2+\gamma_v^2) \le \frac{1}{K-1}$. Applying \eqref{eqa-scaffold-ce} for $k$ times, we obtain
\begin{align*}
    \cE_{i,k}^t &\le \Big(12L^2\gamma_u^2(\sumk\cC_{i,k}^t\! +\! \frac{1}{n}\cC^t) + 6K\gamma_u^2\cG_u^t +2K\gamma_v^2 \|\dv f_i(u^t, v_i^t)\|^2 + 13\sigma_u^2\gamma_u^2 +\sigma_v^2\gamma_v^2\Big) \\
    &\quad \times \sum\limits_{r=0}^{k-1}(1+\frac{2}{K-1})^r\\
    &\le 48L^2K\gamma_u^2(\sumk\cC_{i,k}^t + \frac{1}{n}\cC^t) + 24K^2\gamma_u^2\cG_u^t + 8K^2\gamma_v^2\|\dv f_i(u^t, v_i^t)\|\\
    &\quad + 52K(\sigma_u^2\gamma_u^2+\sigma_v^2\gamma_v^2).
\end{align*}
The final inequality since $\sum_{r=0}^{k-1}\big(1+\frac{2}{K-1}\big)^r \le \sum_{r=0}^{K-1}\big(1+\frac{2}{K-1}\big)^r \le \frac{K-1}{2}\big(1+\frac{2}{K-1}\big)^{K-1}$ and $1+x \le e^x$. Note that $\cE^t = \sumik \cE_{i,k}^t$ and $\cC^t = \sumik \cC_{i,k}^t$, we have
\begin{align*}
    \cE^t &\le 96L^2K^2\gamma_u^2\cC^t + 24nK^3\gamma_u^2\cG_u^t + 8nK^3\gamma_v^2\cG_{\vv}^t+52nK^2(\sigma_u^2\gamma_u^2+\sigma_v^2\gamma_v^2).
\end{align*}
\end{proof}

\begin{proof}[\underline{Proof of Lemma \ref{lem-scaffold-c}}]
Recall that $\cC^t := \sum_{i,k}\EE (\|\alpha_{i,k}^{t-1}-u^t\|^2 + \|\beta_{i,k}^{t-1}-v_i^t\|^2)$. For $k=0,1,...,K-1$, we have
\begin{align*}
\EE \|\alpha_{i,k}^{t-1}-u^t\|^2 
    &=\EE \Big\|(1-\chto)(\alpha_{i,k}^{t-2}-u^t) + \chto (\hat u_{i,k}^{t-1} - u^t)\Big\|^2\\
    &=\left(1-\frac{m}{n}\right)\EE\|\alpha_{i,k}^{t-2}-u^t\|^2 + \frac{m}{n}\EE\|\hat u_{i,k}^{t-1}-u^t\|^2\\
    &\le\left(1-\frac{m}{n}\right)\EE \|\alpha_{i,k}^{t-2}-u^t\|^2+ \frac{2m}{n}\EE(\|\hat u_{i,k}^{t-1}-u^{t-1}\|^2 + \|\Delta u^{t-1}\|^2).
\end{align*}
The first term can be bounded as
\begin{align*}
\EE\|\alpha_{i,k}^{t-2}-u^t\|^2 
    &= \EE\Big(\|\alpha_{i,k}^{t-2}-u^{t-1}\|^2 + \|\Delta u^{t-1}\|^2 + 2\langle \Delta u^{t-1}, \alpha_{i,k}^{t-2}-u^{t-1}\rangle\Big)\\
    &\le \EE\left((1+\kappa)\|\alpha_{i,k}^{t-2}-u^{t-1}\|^2 + \|\Delta u^{t-1}\|^2\right) + \frac{1}{\kappa}\EE\|\EE(\Delta u^{t-1}|\cF^{t-1})\|^2.
\end{align*}
Setting $\kappa$ to be $\frac{m}{2(n-m)}$ and utilizing the inequalities $\frac{2(n-m)}{m} \le \frac{2n}{m}, m\le n$, we have
\begin{align*}
    \EE\|\alpha_{i,k}^{t-1}-u^t\|^2 &\le \Big(1-\frac{m}{2n}\Big)\EE\|\alpha_{i,k}^{t-2}-u^{t-1}\|^2  + 2\EE\|\Delta u^{t-1}\|^2 \\
    &\quad + \frac{2m}{n}\EE\|\hat u_{i,k}^{t-1}-u^{t-1}\|^2 + \frac{2n}{m}\EE\|\EE(\Delta u^{t-1}|\cF^{t-1})\|^2.
\end{align*}
A similar deduction can also be employed to infer that
\begin{align*}
    \EE\|\beta_{i,k}^{t-1}-v_i^t\|^2 &\le \Big(1-\frac{m}{2n}\Big)\EE\|\beta_{i,k}^{t-1}-v_i^{t-1}\|^2 + 2\EE\|\Delta v_i^{t-1}\|^2 \\
    &\quad + \frac{2m}{n}\EE\|\hat v_{i,k}^{t-1}-v_i^{t-1}\|^2 + \frac{2n}{m}\EE\|\EE(\Delta v_i^{t-1}|\cF^{t-1})\|^2.
\end{align*}
Combining the two inequalities presented above, we obtain the following expression:
\begin{equation}\label{ineq-c}
\begin{aligned}
    \mathcal{C}^t &\le \left(1-\frac{m}{2n}\right)\mathcal{C}^{t-1} + 2nK\mathbb{E}\left(\left\|\Delta u^{t-1}\right\|^2 +\frac{1}{n}\sum_{i=1}^{n}\left\|\Delta v_i^{t-1}\right\|^2\right)+ \frac{2m}{n}\mathcal{E}^{t-1}\\
    &\quad + \frac{2n^2K}{m} \EE\bigg(\|\EE(\Delta u^{t-1}|\cF^{t-1})\|^2 + \avgi\|\EE(\Delta v_i^{t-1}|\cF^{t-1})\|^2\bigg).
\end{aligned}
\end{equation}
Note that \(\mathbb{E}(\Delta u^{t-1}|\mathcal{F}^{t-1}) = -\frac{\gamma}{n}\sum_{i,k}\mathbb{E}\Big(\du f_i(\hat u_{i,k}^{t-1}, \hat v_{i,k}^{t-1})|\mathcal{F}^{t-1}\Big)\). Taking into account the definition of \(T_{2,u}\) and utilizing Lemma \ref{lem-T2uv}, we have
\begin{equation}\label{ineq-edu}
\begin{aligned}
\EE\|\mathbb{E}(\Delta u^{t-1}|\mathcal{F}^{t-1})\|^2 &= \frac{\gamma^2}{n^2}\EE\Big\|\EE\Big(\sumik\du f_i(\hat u_{i,k}^{t-1}, \hat v_{i,k}^{t-1})|\cF^{t-1}\Big)\Big\|^2\\
&\le \frac{\gamma^2}{n^2}\EE\Big\|\sumik\du f_i(\hat u_{i,k}^{t-1}, \hat v_{i,k}^{t-1})\Big\|^2 (= T_{2,u})\\
&\le \frac{4L^2K\gamma^2}{n}\cE^{t-1} + 2K^2\gamma^2\cG_u^{t-1}, 
\end{aligned}
\end{equation}
where the first inequality derives from Jensen's inequality. Similarly, it holds that
\begin{equation}\label{ineq-edv}
\begin{aligned}
    \avgi\EE\|\EE(\Delta v_i^{t-1}|\cF^{t-1})\|^2 &\le T_{2,\vv}
    \le \frac{4L^2K\gamma^2m}{n^2}\cE^{t-1} + 2K^2\gamma^2\hat\cG_{\vv}^{t-1}.
\end{aligned}
\end{equation}
Combining with inequalities \eqref{ineq-dus}, \eqref{ineq-dv} as well as Lemmas \ref{lem-T2uv}, \ref{lem-S3u} and recalling the definitions of $T_{2,u}, T_{2,\vv}, S_{3,u}$, we have
\begin{equation}\label{ineq-duvs}
\begin{aligned}
\EE\|\Delta u^{t-1}\|^2 &\le 4K^2\gamma^2\cG_u^{t-1}\! +\! \frac{72L^2K\gamma^2}{n}\cE^{t-1}\! + \!\frac{64L^2K\gamma^2(n-m)}{mn^2}\cC^{t-1} \!+\! \frac{18K\sigma_u^2\gamma^2}{m},\\
\avgi \EE \|\Delta v_i^{t-1}\|^2 &\le 4K^2\gamma^2 \hat\cG_{\vv}^{t-1} + \frac{8L^2K\gamma^2m}{n^2}\cE^{t-1}  + \frac{2K\sigma_v^2\gamma^2m}{n}.
\end{aligned}
\end{equation}
By substituting \eqref{ineq-edu}, \eqref{ineq-edv} and \eqref{ineq-duvs} into the expression \eqref{ineq-c} and applying $72LK\gamma \le \min(\frac{m}{n^{2/3}},1)$, we obtain 
\begin{equation}\label{ineq-scaffold-c}
\begin{aligned}
    \cC^t 
    &\le \left(1-\frac{m}{3n}\right)\cC^{t-1} \!+ \frac{3m}{n^{1/3}}\cE^{t-1}  + \frac{mn^{2/3}K}{72L^2}\left(\cG_u^{t-1} + \hat \cG_{\vv}^{t-1}\right)+ \frac{m^2}{72L^2}\Big(\frac{\sigma_u^2}{m} + \frac{m\sigma_v^2}{n}\Big).
\end{aligned}
\end{equation}
From Lemma \ref{lem-scaffold-e} and $\eta_u \ge \sqrt{m},\ \eta_v \ge \sqrt{\frac{n}{m}}$, we have
\begin{align*}
    \cE^t
    &\le 96L^2K^2\gamma_u^2\cC^t + 24nK^3\gamma_u^2\cG_u^t + 8nK^3\gamma_v^2\cG_{\vv}^t +52nK^2(\sigma_u^2\gamma_u^2+\sigma_v^2\gamma_v^2)\\
    &\le  \frac{96L^2K^2\gamma^2}{\eta_u^2}\cC^t + 24nK^3\gamma^2\Big(\frac{1}{\eta_u^2}\cG_u^t + \frac{1}{\eta_v^2}\cG_{\vv}^t\Big) +52nK^2\gamma^2\Big(\frac{\sigma_u^2}{\eta_u^2}+\frac{\sigma_v^2}{\eta_v^2}\Big)\\
    &\le \frac{1}{54n^{2/3}}\cC^t + \frac{nK}{216L^2}\left(\cG_u^t + \hat \cG_{\vv}^t\right) + \frac{mn^{1/3}}{72L^2}\Big(\frac{\sigma_u^2}{m}+\frac{m\sigma_v^2}{n}\Big),
\end{align*}
The final inequality derives from $72LK\gamma \le \min(\frac{m}{n^{2/3}},1)$. Combining this result with \eqref{ineq-scaffold-c}, we complete the proof.
\end{proof}

\bibliographystyle{siam}
\bibliography{references}

\end{document}

%% file: preamble.tex
\usepackage{amsmath,amsfonts,amssymb}
\usepackage{algpseudocode}
\usepackage{algorithm}
\usepackage{array}
\usepackage[caption=false,font=scriptsize,labelfont=rm,textfont=rm]{subfig}
\usepackage{textcomp}
\usepackage{stfloats}
\usepackage{url}
\usepackage{verbatim}
\usepackage{graphicx}
\usepackage{epstopdf}
\usepackage{multirow}
\usepackage{float}
\usepackage{color}
\usepackage{bbm}
\usepackage{adjustbox}
\usepackage{cleveref}
\usepackage[normalem]{ulem}
\usepackage{enumitem}
\usepackage{booktabs}
\usepackage{threeparttable}
\allowdisplaybreaks
\usepackage{multirow}

\newcommand{\vv}{{\mathbf{v}}}

\newcommand{\cB}{{\mathcal{B}}}
\newcommand{\cC}{{\mathcal{C}}}

\newcommand{\cE}{{\mathcal{E}}}
\newcommand{\cF}{{\mathcal{F}}}
\newcommand{\cG}{{\mathcal{G}}}

\newcommand{\cS}{{\mathcal{S}}}

\newcommand{\EE}{\mathbb{E}}
\newcommand{\RR}{\mathbb{R}}

\newcommand{\ch}{{\mathbbm{1}}}
\newcommand{\cht}{{\mathbbm{1}}_{\{i \in \mathcal{S}^t\}}}
\newcommand{\chto}{{\mathbbm{1}}_{\{i \in \mathcal{S}^{t-1}\}}}

\newcommand{\sumi}{\sum\limits_{i=1}^n}

\newcommand{\sumk}{\sum\limits_{k=0}^{K-1}}
\newcommand{\sumik}{\sum\limits_{i,k}}

\newcommand{\sumpit}{\sum\limits_{i \in \cS^t}}
\newcommand{\avgk}{\frac{1}{K}\sum\limits_{k=0}^{K-1}}
\newcommand{\avgi}{\frac{1}{n}\sum\limits_{i=1}^n}
\newcommand{\avgj}{\frac{1}{n}\sum\limits_{j=1}^n}

\newcommand{\avgt}{\frac{1}{T}\sum\limits_{t=0}^{T-1}}
\newcommand{\avgik}{\frac{1}{nK}\sum\limits_{i,k}}

\newcommand{\du}{\nabla_u}
\newcommand{\dv}{\nabla_v}

\let\oldlangle\langle
\let\oldrangle\rangle
\renewcommand{\langle}{\left\oldlangle}
\renewcommand{\rangle}{\right\oldrangle}

\newtheorem{assumption}[theorem]{Assumption}

\definecolor{darkgreen}{rgb}{0,0.7,0}

%% file: ex_arxiv.bbl
\begin{thebibliography}{10}

\bibitem{antoniadis2011penalized}
{\sc A.~Antoniadis, I.~Gijbels, and M.~Nikolova}, {\em Penalized likelihood
  regression for generalized linear models with non-quadratic penalties},
  Annals of the Institute of Statistical Mathematics, 63 (2011).

\bibitem{chayti2021linear}
{\sc E.~M. Chayti, S.~P. Karimireddy, S.~U. Stich, N.~Flammarion, and
  M.~Jaggi}, {\em Linear speedup in personalized collaborative learning}, arXiv
  preprint arXiv:2111.05968,  (2021).

\bibitem{cheng2023momentum}
{\sc Z.~Cheng, X.~Huang, and K.~Yuan}, {\em Momentum benefits non-iid federated
  learning simply and provably}, arXiv preprint arXiv:2306.16504,  (2023).

\bibitem{collins2021exploiting}
{\sc L.~Collins, H.~Hassani, A.~Mokhtari, and S.~Shakkottai}, {\em Exploiting
  shared representations for personalized federated learning}, in International
  conference on machine learning, PMLR, 2021, pp.~2089--2099.

\bibitem{deng2020adaptive}
{\sc Y.~Deng, M.~M. Kamani, and M.~Mahdavi}, {\em Adaptive personalized
  federated learning}, arXiv preprint arXiv:2003.13461,  (2020).

\bibitem{durrett2019probability}
{\sc R.~Durrett}, {\em Probability: theory and examples}, vol.~49, Cambridge
  university press, 2019.

\bibitem{el2022personalized}
{\sc K.~B. El~Houcine~Bergou, A.~Dutta, and P.~Richt{\'a}rik}, {\em
  Personalized federated learning with communication compression}, arXiv
  preprint arXiv:2209.05148,  (2022).

\bibitem{fallah2020personalized}
{\sc A.~Fallah, A.~Mokhtari, and A.~Ozdaglar}, {\em Personalized federated
  learning with theoretical guarantees: A model-agnostic meta-learning
  approach}, Advances in Neural Information Processing Systems, 33 (2020),
  pp.~3557--3568.

\bibitem{huang2023distributed}
{\sc K.~Huang, X.~Li, and S.~Pu}, {\em Distributed stochastic optimization
  under a general variance condition}, arXiv preprint arXiv:2301.12677,
  (2023).

\bibitem{huang2022stochastic}
{\sc Y.~Huang, J.~Xu, W.~Meng, and H.-T. Wai}, {\em Stochastic gradient
  tracking methods for distributed personalized optimization over networks}, in
  IEEE Conference on Decision and Control (CDC), IEEE, 2022, pp.~4571--4578.

\bibitem{kairouz2021advances}
{\sc P.~e.~a. Kairouz}, {\em Advances and open problems in federated learning},
  Foundations and Trends in Machine Learning, 14 (2021), pp.~1--210.

\bibitem{karimireddy2020scaffold}
{\sc S.~P. Karimireddy, S.~Kale, M.~Mohri, S.~Reddi, S.~Stich, and A.~T.
  Suresh}, {\em Scaffold: Stochastic controlled averaging for federated
  learning}, in International Conference on Machine Learning, PMLR, 2020,
  pp.~5132--5143.

\bibitem{konevcny2016federated}
{\sc J.~Kone{\v{c}}n{\`y}, H.~B. McMahan, F.~X. Yu, P.~Richt{\'a}rik, A.~T.
  Suresh, and D.~Bacon}, {\em Federated learning: Strategies for improving
  communication efficiency}, arXiv preprint arXiv:1610.05492,  (2016).

\bibitem{li2021ditto}
{\sc T.~Li, S.~Hu, A.~Beirami, and V.~Smith}, {\em Ditto: Fair and robust
  federated learning through personalization}, in International Conference on
  Machine Learning, PMLR, 2021, pp.~6357--6368.

\bibitem{li2019convergence}
{\sc X.~Li, K.~Huang, W.~Yang, S.~Wang, and Z.~Zhang}, {\em On the convergence
  of fedavg on non-iid data}, arXiv preprint arXiv:1907.02189,  (2019).

\bibitem{mcmahan2017communication}
{\sc B.~McMahan, E.~Moore, D.~Ramage, S.~Hampson, and B.~A. y~Arcas}, {\em
  Communication-efficient learning of deep networks from decentralized data},
  in Artificial intelligence and statistics, PMLR, 2017, pp.~1273--1282.

\bibitem{mushtaq2021spider}
{\sc E.~Mushtaq, C.~He, J.~Ding, and S.~Avestimehr}, {\em Spider: Searching
  personalized neural architecture for federated learning}, arXiv preprint
  arXiv:2112.13939,  (2021).

\bibitem{pillutla2022federated}
{\sc K.~Pillutla, K.~Malik, A.-R. Mohamed, M.~Rabbat, M.~Sanjabi, and L.~Xiao},
  {\em Federated learning with partial model personalization}, in International
  Conference on Machine Learning, PMLR, 2022, pp.~17716--17758.

\bibitem{qu2021federated}
{\sc Z.~Qu, K.~Lin, Z.~Li, and J.~Zhou}, {\em Federated learning’s blessing:
  Fedavg has linear speedup}, in ICLR 2021-Workshop on Distributed and Private
  Machine Learning (DPML), 2021.

\bibitem{t2020personalized}
{\sc C.~T~Dinh, N.~Tran, and J.~Nguyen}, {\em Personalized federated learning
  with moreau envelopes}, Advances in Neural Information Processing Systems, 33
  (2020), pp.~21394--21405.

\bibitem{wang2022fedadmm}
{\sc H.~Wang, S.~Marella, and J.~Anderson}, {\em Fedadmm: A federated
  primal-dual algorithm allowing partial participation}, in 2022 IEEE 61st
  Conference on Decision and Control (CDC), IEEE, 2022, pp.~287--294.

\bibitem{wei2023personalized}
{\sc K.~Wei, J.~Li, C.~Ma, M.~Ding, W.~Chen, J.~Wu, M.~Tao, and H.~V. Poor},
  {\em Personalized federated learning with differential privacy and
  convergence guarantee}, IEEE Transactions on Information Forensics and
  Security,  (2023).

\bibitem{xin2021improved}
{\sc R.~Xin, U.~A. Khan, and S.~Kar}, {\em An improved convergence analysis for
  decentralized online stochastic non-convex optimization}, IEEE Transactions
  on Signal Processing, 69 (2021), pp.~1842--1858.

\bibitem{yang2021achieving}
{\sc H.~Yang, M.~Fang, and J.~Liu}, {\em Achieving linear speedup with partial
  worker participation in non-iid federated learning}, arXiv preprint
  arXiv:2101.11203,  (2021).

\bibitem{yu2019parallel}
{\sc H.~Yu, S.~Yang, and S.~Zhu}, {\em Parallel restarted {SGD} with faster
  convergence and less communication: Demystifying why model averaging works
  for deep learning}, in Proceedings of the AAAI Conference on Artificial
  Intelligence, vol.~33, 2019, pp.~5693--5700.

\bibitem{zhou2023federated}
{\sc S.~Zhou and G.~Y. Li}, {\em Federated learning via inexact {ADMM}}, IEEE
  Transactions on Pattern Analysis and Machine Intelligence,  (2023).

\end{thebibliography}
